\documentclass[svgnames,11pt,a4paper,reqno]{amsart}
\usepackage[left=28mm, right=28mm, bottom=25mm, top=25mm,includefoot]{geometry}

\usepackage{mathrsfs}
\usepackage{amsmath}
\usepackage{amsthm}
\usepackage{tikz}
\usetikzlibrary{
knots,
hobby,
decorations.pathreplacing,
shapes.geometric,
calc,
decorations.markings
}
\usepgfmodule{decorations}
\usepackage{float}
\usepackage{setspace}
\usepackage{tikz-cd}
\usetikzlibrary{braids}
\usepackage{wrapfig}
\usepackage{lipsum}
\usepackage{marginnote}
\definecolor{invisible}{RGB}{245,245,245}
\usepackage[tableposition=below]{caption}
\usepackage{subcaption}

\usepackage{xspace}

\usepackage{aligned-overset}

\usepackage{mathtools}

\usepackage{scalerel,stackengine}
\stackMath
\newcommand\reallywidehat[1]{%
\savestack{\tmpbox}{\stretchto{%
\scaleto{%
\scalerel*[\widthof{\ensuremath{#1}}]{\kern-.6pt\bigwedge\kern-.6pt}%
{\rule[-\textheight/2]{1ex}{\textheight}}%
}{\textheight}%
}{0.5ex}}%
\stackon[1pt]{#1}{\tmpbox}%
}

\usetikzlibrary{matrix}

\usepackage{array}

\usepackage{graphicx}
\usepackage[all]{xy}

\theoremstyle{definition}
\newtheorem{theorem}{Theorem}[section]

\newtheorem{conjecture}[theorem]{Conjecture}
\newtheorem{lemma}[theorem]{Lemma}
\newtheorem{claim}{Claim}
\newtheorem{observation}[theorem]{Observation}
\newtheorem{proposition}[theorem]{Proposition}
\newtheorem{corollary}[theorem]{Corollary}

\theoremstyle{definition}

\newtheorem{definition}[theorem]{Definition}

\numberwithin{equation}{section}
\makeatletter
\newcommand*\PrintSkips[1]{%
\typeout{In #1:}%
\typeout{\@spaces above: \the\abovecaptionskip}%
\typeout{\@spaces below: \the\belowcaptionskip}%
}
\makeatother

\usepackage{amssymb}
\usepackage{xcolor}

\definecolor{invisible}{RGB}{245,245,245}
\definecolor{CelestialBlue}{rgb}{0.29, 0.59, 0.82}
\definecolor{brickred}{rgb}{0.8, 0.25, 0.33}
\definecolor{amaranth}{rgb}{0.9, 0.17, 0.31}
\definecolor{awesome}{rgb}{1.0, 0.13, 0.32}
\definecolor{linkcolor}{HTML}{0000FF} %
\definecolor{citecolor}{HTML}{0000FF} %
\definecolor{urlcolor}{HTML}{0000FF} %
\usepackage[colorlinks, urlcolor=awesome, citecolor=awesome]{hyperref}
\hypersetup{pdfstartview=FitH, linkcolor=MidnightBlue, urlcolor=urlcolor, colorlinks=true}

\newcommand\restr[2]{{
  \left.\kern-\nulldelimiterspace 
  #1 
  \vphantom{\big|} 
  \right|_{#2} 
  }}
\usepackage[shortlabels]{enumitem}

\let\oldbibliography\thebibliography
\renewcommand{\thebibliography}[1]{
 \oldbibliography{#1}
 \setlength{\itemsep}{0pt}
}

\newcommand{\M}{\mathcal{M}}
\newcommand{\D}{\mathcal{D}}
\newcommand{\U}{\mathcal{U}}

\renewcommand{\S}{\mathcal{S}^2}
\renewcommand{\L}{\mathcal{L}}

\renewcommand{\P}{\mathcal{P}}
\newcommand{\J}{\mathcal{V}}
\newcommand{\W}{\mathcal{W}}
\newcommand{\V}{\mathcal{V}}
\newcommand{\Q}{\mathcal{Q}}

\newcommand{\T}{\mathcal{T}}

\RequirePackage[capitalize]{cleveref}

\newcommand*{\captionsource}[2]{%
  \caption[{#1}]{%
    #1 %
    #2%
  }%
}

\begin{document}

\title[\resizebox{5.3in}{!}{Sharpness of the Morton-Franks-Williams inequality for positive knots and links}]{Sharpness of the Morton-Franks-Williams\\ inequality for positive knots and links \ \vspace{-0.2cm}}
\author{Ilya Alekseev}
\thanks{This work was supported by the Russian Science Foundation under grant no. 22-11-00299, \\ https://rscf.ru/en/project/22-11-00299/.} 
\address{Steklov Mathematical Institute of Russian Academy of Sciences, 8 Gubkina St., Moscow 119991, Russia}
\email{ilyaalekseev@yahoo.com}

\begin{abstract}
We provide a combinatorial characterisation of positive diagrams satisfying the equality in the Morton-Franks-Williams bound for the degrees of the HOMFLY-PT polynomial. This characterisation allows generating with relative ease examples of diagrams realizing the crossing number, the braid index, and the maximal self-linking number. Besides, we suggest a conjecture concerning the sharpness of the Morton-Franks-Williams inequality for strongly quasipositive links.
\end{abstract}

\maketitle

\ \vspace{-1.1cm}

\section{Introduction}

In this paper, we elaborate on recent developments in polynomial invariants and obtain new applications for complexity measures of knots and links given by means of their diagrams.

Historically, complexity measures play a leading part in knot enumeration, dating back to the end of the 19th century, when P. G. Tate, T. P. Kirkman, and C. N. Little 
tabulated the simplest knots. For this, they followed some heuristic postulates, known as Tait's conjectures, related to the ‘tightness’ of special diagram classes (see \cite{Hos05}).

Typically, link projections reveal little immediate information about the links. For instance, an outwardly complicated diagram may represent the unknot. However, certain prominent diagram types, such as alternating\footnote{A diagram is {\it alternating} if, as one travels along each of its components, one meets ‘over’ and ‘under’ crossings alternately.} and positive\footnote{
An oriented diagram is {\it positive} if any of its crossings is positive.}, force a link to have much more structure. In particular, the first Tait conjecture pertains to the minimality of alternating diagrams with respect to the crossing number. It was not until the renowned discovery of the Jones polynomial in the mid-80s that the conjecture was proved. Our primary goal is to obtain comparable results for positive diagrams.

According to L. Kauffman, K. Murasugi, and M. Thistlethwaite~\cite{Kau87,Mur87,Thi87}, the breadth of the Jones polynomial $\J \in \mathbb{Z}[t^{1/2},t^{-1/2}]$ provides the following estimate for the crossing number of a link diagram:
\begin{align}\label{JonesCrossingNumberBound}
\tag{KMT}
{\rm deg}^+ \J(\D) - {\rm deg}^- \J(\D) \leqslant {\rm Cr}(\D).
\end{align}
For alternating diagrams without isthmuses\footnote{An {\it isthmus} is a crossing of a diagram whose deletion disconnects it.}, the equality holds. Therefore, such diagrams realize the minimum crossing number of the corresponding links, which is the statement of Tait's conjecture.

In fact, the sharpness of the Kauffman-Murasugi-Thistlethwaite inequality is exhausted by connected sums of alternating diagrams without isthmuses (see \cite{Tur87}), so this approach {\it a priori} fails to establish the crossing number of non-alternating links. However, certain two-variable generalizations of the Jones polynomial provide alternative estimates for the crossing number.

For example, the Kauffman polynomial directly extends the bound \eqref{JonesCrossingNumberBound} and provides a similar clear criterion for the sharpness, leading to a broader prominent class of the so-called adequate diagrams (see \cite{Thi88}). 

In this paper, we concentrate on a bound coming from the HOMFLY-PT polynomial (see below). In this case, the sharpness implies more rigid minimality conditions, and the criterion is far from a full description.

The HOMFLY-PT polynomial $\P \in \mathbb{Z}[a^{\pm 1},z^{\pm 1}]$, defined for oriented links, provides the following estimate for the crossing number \cite{Mur91,Gru03,Di04}:
\begin{align}\label{MFWCrossingNumberBound}
\tag{MFW}
{\rm deg}_z^+\P(\D) + \dfrac{1}{2}\left({\rm deg}_a^+\P(\D) - {\rm deg}_a^-\P(\D)\right) \leqslant {\rm Cr}(\D).
\end{align}
Our main objective is to establish which diagrams satisfy the equality in \eqref{MFWCrossingNumberBound}.

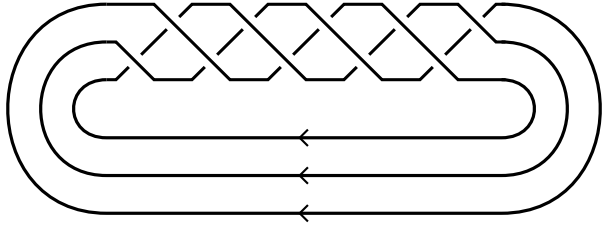
\begin{figure}[H]
\centering
\begin{tikzpicture}[rotate=90,scale=0.385, every node/.style={scale=0.5}]
\pic[
  rotate=90,
  line width=1.15pt,
  braid/control factor=0,
  braid/nudge factor=0,
  braid/gap=0.11,
  braid/number of strands = 3,
  name prefix=braid,
] at (0,0) {braid={
s_1^{-1}s_2^{-1}s_1^{-1}s_2^{-1}s_1^{-1}s_2^{-1}s_1^{-1}s_2^{-1}s_1^{-1}s_2^{-1}
}};
\draw[very thick ] (0,0) .. controls ++(0,1.5) and ++(0,1.5) .. (-2,0)
    -- (-2,-13.5) .. controls ++(0,-1.5) and ++(0,-1.5) .. (0,-13.5);
\draw [thick,-Straight Barb] (-2,-6.65) -- (-2,-6.55);
\draw[very thick] (1.3,0) .. controls ++(0,3) and ++(0,3) .. (-3.3,0) 
-- (-3.3,-13.5) .. controls ++(0,-3) and ++(0,-3) .. (1.3,-13.5);
\draw [thick,-Straight Barb] (-3.3,-6.65) -- (-3.3,-6.55);
\draw[very thick] (2.6,0) .. controls ++(0,4.5) and ++(0,4.5) .. (-4.6,0)
    -- (-4.6,-13.5) .. controls ++(0,-4.5) and ++(0,-4.5) .. (2.6,-13.5);
\draw [thick,-Straight Barb] (-4.6,-6.65) -- (-4.6,-6.55);
\end{tikzpicture}
\caption{A standard diagram of the $(3,5)$ torus knot.}
\label{TorusKnot}
\end{figure}

For example, if $\D$ is the diagram shown in Fig. \ref{TorusKnot}, then \eqref{MFWCrossingNumberBound} is sharp since
\begin{align*}
\P(\D) = (z^8+8z^6+21z^4+21z^2+7)a^8-(z^6+7z^4+14z^2+8)a^{10}+(z^2+2)a^{12}.
\end{align*}
This data is displayed in Fig. \ref{CoefficientsOfTorus}.

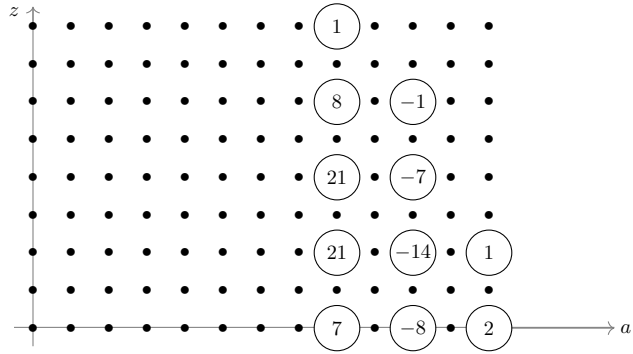
\begin{figure}[H]
\centering
\begin{tikzpicture}[rotate=90,scale=0.5, every node/.style={scale=0.7}]
\node at (8+0.4,0.5) {$z$};
\node at (0,-13.6-2) {$a$};
\draw[draw=gray,-To] (0,-12) -- (0,-13.3-2);
\draw[draw=gray,-To] (8,0) -- (8+0.02+0.5,0);
\draw[draw=gray] (0+0.02,0.5) -- (0+0.02,-13.3-2);
\draw[draw=gray] (-0.5+0.02,0) -- (8+0.02+0.5,0);
\node at (0,0) {$\bullet$};
\node at (0,-1) {$\bullet$};
\node at (0,-2) {$\bullet$};
\node at (0,-3) {$\bullet$};
\node at (0,-4) {$\bullet$};
\node at (0,-5) {$\bullet$};
\node at (0,-6) {$\bullet$};
\node at (0,-7) {$\bullet$};
\filldraw[fill=white] (0,-8) circle (0.6); \node at (0,-8) {$7$};
\node at (0,-9) {$\bullet$};
\filldraw[fill=white] (0,-10) circle (0.6); \node at (0,-10) {$-8$};
\node at (0,-11) {$\bullet$};
\filldraw[fill=white] (0,-12) circle (0.6); \node at (0,-12) {$2$};
\node at (1,0) {$\bullet$};
\node at (1,-1) {$\bullet$};
\node at (1,-2) {$\bullet$};
\node at (1,-3) {$\bullet$};
\node at (1,-4) {$\bullet$};
\node at (1,-5) {$\bullet$};
\node at (1,-6) {$\bullet$};
\node at (1,-7) {$\bullet$};
\node at (1,-8) {$\bullet$};
\node at (1,-9) {$\bullet$};
\node at (1,-10) {$\bullet$};
\node at (1,-11) {$\bullet$};
\node at (1,-12) {$\bullet$};
\node at (2,0) {$\bullet$};
\node at (2,-1) {$\bullet$};
\node at (2,-2) {$\bullet$};
\node at (2,-3) {$\bullet$};
\node at (2,-4) {$\bullet$};
\node at (2,-5) {$\bullet$};
\node at (2,-6) {$\bullet$};
\node at (2,-7) {$\bullet$};
\filldraw[fill=white] (2,-8) circle (0.6); \node at (2,-8) {$21$};
\node at (2,-9) {$\bullet$};
\filldraw[fill=white] (2,-10) circle (0.6); \node at (2,-10) {$-14$};
\node at (2,-11) {$\bullet$};
\filldraw[fill=white] (2,-12) circle (0.6); \node at (2,-12) {$1$};
\node at (3,0) {$\bullet$};
\node at (3,-1) {$\bullet$};
\node at (3,-2) {$\bullet$};
\node at (3,-3) {$\bullet$};
\node at (3,-4) {$\bullet$};
\node at (3,-5) {$\bullet$};
\node at (3,-6) {$\bullet$};
\node at (3,-7) {$\bullet$};
\node at (3,-8) {$\bullet$};
\node at (3,-9) {$\bullet$};
\node at (3,-10) {$\bullet$};
\node at (3,-11) {$\bullet$};
\node at (3,-12) {$\bullet$};
\node at (4,0) {$\bullet$};
\node at (4,-1) {$\bullet$};
\node at (4,-2) {$\bullet$};
\node at (4,-3) {$\bullet$};
\node at (4,-4) {$\bullet$};
\node at (4,-5) {$\bullet$};
\node at (4,-6) {$\bullet$};
\node at (4,-7) {$\bullet$};
\filldraw[fill=white] (4,-8) circle (0.6); \node at (4,-8) {$21$};
\node at (4,-9) {$\bullet$};
\filldraw[fill=white] (4,-10) circle (0.6); \node at (4,-10) {$-7$};
\node at (4,-11) {$\bullet$};
\node at (4,-12) {$\bullet$};
\node at (5,0) {$\bullet$};
\node at (5,-1) {$\bullet$};
\node at (5,-2) {$\bullet$};
\node at (5,-3) {$\bullet$};
\node at (5,-4) {$\bullet$};
\node at (5,-5) {$\bullet$};
\node at (5,-6) {$\bullet$};
\node at (5,-7) {$\bullet$};
\node at (5,-8) {$\bullet$};
\node at (5,-9) {$\bullet$};
\node at (5,-10) {$\bullet$};
\node at (5,-11) {$\bullet$};
\node at (5,-12) {$\bullet$};
\node at (6,0) {$\bullet$};
\node at (6,-1) {$\bullet$};
\node at (6,-2) {$\bullet$};
\node at (6,-3) {$\bullet$};
\node at (6,-4) {$\bullet$};
\node at (6,-5) {$\bullet$};
\node at (6,-6) {$\bullet$};
\node at (6,-7) {$\bullet$};
\filldraw[fill=white] (6,-8) circle (0.6); \node at (6,-8) {$8$};
\node at (6,-9) {$\bullet$};
\filldraw[fill=white] (6,-10) circle (0.6); \node at (6,-10) {$-1$};
\node at (6,-11) {$\bullet$};
\node at (6,-12) {$\bullet$};
\node at (7,0) {$\bullet$};
\node at (7,-1) {$\bullet$};
\node at (7,-2) {$\bullet$};
\node at (7,-3) {$\bullet$};
\node at (7,-4) {$\bullet$};
\node at (7,-5) {$\bullet$};
\node at (7,-6) {$\bullet$};
\node at (7,-7) {$\bullet$};
\node at (7,-8) {$\bullet$};
\node at (7,-9) {$\bullet$};
\node at (7,-10) {$\bullet$};
\node at (7,-11) {$\bullet$};
\node at (7,-12) {$\bullet$};
\node at (8,0) {$\bullet$};
\node at (8,-1) {$\bullet$};
\node at (8,-2) {$\bullet$};
\node at (8,-3) {$\bullet$};
\node at (8,-4) {$\bullet$};
\node at (8,-5) {$\bullet$};
\node at (8,-6) {$\bullet$};
\node at (8,-7) {$\bullet$};
\filldraw[fill=white] (8,-8) circle (0.6); \node at (8,-8) {$1$};
\node at (8,-9) {$\bullet$};
\node at (8,-10) {$\bullet$};
\node at (8,-11) {$\bullet$};
\node at (8,-12) {$\bullet$};
\end{tikzpicture}
\caption{The HOMFLY-PT polynomial of the $(3,5)$ torus knot.}
\label{CoefficientsOfTorus}
\end{figure}

To clarify the origin of the estimate \eqref{MFWCrossingNumberBound}, we define the following diagram quantities. The difference between the number of positive 
\begin{tikzpicture}[rotate=90, scale=0.9, baseline=1]
\draw[->] (0.3,0.3)--(0,0);
\draw[thick] (0.3,0.3)--(0,0);
\draw[thick] (0.2,0.1)--(0.3,0);
\draw[->] (0.2,0.1)--(0.3,0);
\draw[thick] (0,0.3)--(0.1,0.2);
\end{tikzpicture}
and the number of negative 
\begin{tikzpicture}[rotate=90, scale=0.9, baseline=1]
\draw[->] (0,0.3)--(0.3,0);
\draw[thick] (0,0.3)--(0.3,0);
\draw[thick] (0.1,0.1)--(0,0);
\draw[->] (0.1,0.1)--(0,0);
\draw[thick] (0.3,0.3)--(0.2,0.2);
\end{tikzpicture}
crossings of a diagram $\D$ is called the {\it writhe} and is denoted by $\omega(\D)$. Diagram transformations of the form~\begin{tikzpicture}[rotate=90, scale=0.9, baseline=1]
\draw[->] (0.3,0.3)--(0,0);
\draw[thick] (0.3,0.3)--(0,0);
\draw[thick] (0.2,0.1)--(0.3,0);
\draw[->] (0.2,0.1)--(0.3,0);
\draw[thick] (0,0.3)--(0.1,0.2);
\end{tikzpicture}
$\mapsto$ 
\begin{tikzpicture}[rotate=90, scale=0.9, baseline=1]
\draw[->] (0,0.3)--(0,0);
\draw[thick] (0,0.3)--(0,0);
\draw[->] (0.3,0.3)--(0.3,0);
\draw[thick] (0.3,0.3)--(0.3,0);
\end{tikzpicture}
$\mathrel{\reflectbox{\ensuremath{\mapsto}}}$
\begin{tikzpicture}[rotate=90, scale=0.9, baseline=1]
\draw[->] (0,0.3)--(0.3,0);
\draw[thick] (0,0.3)--(0.3,0);
\draw[thick] (0.1,0.1)--(0,0);
\draw[->] (0.1,0.1)--(0,0);
\draw[thick] (0.3,0.3)--(0.2,0.2);
\end{tikzpicture}
are called {\it crossing smoothing}. Simple closed curves obtained by smoothing all crossings are called {\it Seifert circles}. The number of Seifert circles of $\D$ is denoted by~$s(\D)$.

As a rule, complexity of any link polynomial satisfying a skein relation is bounded by complexity of any link diagram. 
The HOMFLY-PT polynomial obeys the following ``upper'' (see \cite{Mor86}), ``left'', and ``right'' (see \cite{Mor86, FW87}) inequalities:
\begin{align}
\tag{U}
&{\rm deg}_z^+\P(\D) \leqslant {\rm Cr}(\D) - s(\D) + 1, \label{U} \\
\tag{L}
&\omega(\D) - s(\D)+1 \leqslant {\rm deg}_a^-\P(\D), \label{L} \\
\tag{R}
&{\rm deg}_a^+\P(\D) \leqslant \omega(\D) + s(\D)-1. \label{R}
\end{align}
Note that~\eqref{U}, \eqref{L}, and~\eqref{R} imply \eqref{MFWCrossingNumberBound} and that \eqref{MFWCrossingNumberBound} is sharp if and only if each of~\eqref{U}, \eqref{L}, and~\eqref{R} is sharp. Therefore, to establish and clarify the sharpness, we can look into these bounds separately.

It turns out that the above inequalities relate to other link complexity measures, which are independent of the crossing number. Namely, \eqref{U} provides a bound for the canonical genus, \eqref{L} provides a bound for the maximal self-linking number, and the inequalities \eqref{L} and \eqref{R} provide a bound for the braid index.

\subsection{Main results}

For positive diagrams, the equalities in~\eqref{U} (see~\cite{Cro89}) and in~\eqref{L} (see~\cite{Yok92}) always hold. Therefore, the sharpness of~\eqref{MFWCrossingNumberBound} is equivalent to the sharpness of~\eqref{R}. In this paper, we prove the following combinatorial description for the latter.

\begin{theorem}\label{MainTheorem}
For a positive diagram $\D$, the equality in \eqref{R} holds if and only if $\D$ can be obtained from a zero-crossing diagram by transformations shown in Fig.~\ref{Transformations}.
\end{theorem}

\begin{figure}[H]
\centering
\begin{tikzpicture}[rotate=90, scale=2, baseline=1]
\draw[->] (0.15+0,0.6)--(0.15+0,0);
\draw[thick] (0.15+0,0.6)--(0.15+0,0);
\draw[->] (0.15+0.3,0.6)--(0.15+0.3,0);
\draw[thick] (0.15+0.3,0.6)--(0.15+0.3,0);
\draw[-Latex]   (0.3,-0.15) -- (0.3,-0.45);
\draw[thick] (0.15+0.3,0.3-1.2)--(0.15+0,0-1.2);
\draw[->] (0.15+0.3,0.3-1.2)--(0.15+0,0-1.2);
\draw[thick] (0.15+0.2,0.1-1.2)--(0.15+0.3,0-1.2);
\draw[->] (0.15+0.2,0.1-1.2)--(0.15+0.3,0-1.2);
\draw[thick] (0.15+0,0.3-1.2)--(0.15+0.1,0.2-1.2);
\draw[thick] (0.15+0.3,0.6-1.2)--(0.15+0,0.3-1.2);
\draw[thick] (0.15+0.2,0.4-1.2)--(0.15+0.3,0.3-1.2);
\draw[thick] (0.15+0,0.6-1.2)--(0.15+0.1,0.5-1.2);
\end{tikzpicture}\hspace{1.5cm}
\begin{tikzpicture}[rotate=90, scale=2, baseline=1]
\draw[->] (0.15+0.3,0.3)--(0.15+0,0);
\draw[thick] (0.15+0.3,0.3)--(0.15+0,0);
\draw[thick] (0.15+0.2,0.1)--(0.15+0.3,0);
\draw[->] (0.15+0.2,0.1)--(0.15+0.3,0);
\draw[thick] (0.15+0,0.3)--(0.15+0.1,0.2);
\draw[-Latex]   (0.3,-0.15) -- (0.3,-0.45);
\draw[thick] (0.15+0.3,0.3-1.2)--(0.15+0,0-1.2);
\draw[->] (0.15+0.3,0.3-1.2)--(0.15+0,0-1.2);
\draw[thick] (0.15+0.2,0.1-1.2)--(0.15+0.3,0-1.2);
\draw[->] (0.15+0.2,0.1-1.2)--(0.15+0.3,0-1.2);
\draw[thick] (0.15+0,0.3-1.2)--(0.15+0.1,0.2-1.2);
\draw[thick] (0.15+0.3,0.6-1.2)--(0.15+0,0.3-1.2);
\draw[thick] (0.15+0.2,0.4-1.2)--(0.15+0.3,0.3-1.2);
\draw[thick] (0.15+0,0.6-1.2)--(0.15+0.1,0.5-1.2);
\end{tikzpicture}\hspace{1.5cm}
\begin{tikzpicture}[rotate=90, scale=2, baseline=1]
\draw[thick] (0.3,0.3)--(0,0);
\draw[thick] (0.2,0.1)--(0.3,0);
\draw[thick] (0,0.3)--(0.1,0.2);
\draw[thick] (0.3+0.3,0.3+0.3)--(0.3+0,0.3+0);
\draw[thick] (0.3+0.2,0.3+0.1)--(0.3+0.3,0.3+0);
\draw[thick] (0.3+0,0.3+0.3)--(0.3+0.1,0.3+0.2);
\draw[thick] (0.3,0.6+0.3)--(0,0.6+0);
\draw[thick] (0.2,0.6+0.1)--(0.3,0.6+0);
\draw[thick] (0,0.6+0.3)--(0.1,0.6+0.2);
\draw[thick] (0,0.3)--(0,0.6);
\draw[thick] (0.6,0)--(0.6,0.3);
\draw[thick] (0.6,0.6)--(0.6,0.9);
\draw[-Latex]   (0.3,-0.15) -- (0.3,-0.45);
\draw[-Latex]   (0.3,-0.45) -- (0.3,-0.15);
\draw[thick] (0+0.6,0.3-1.5)--(0+0.6,0.6-1.5);
\draw[thick] (0+0,0-1.5)--(0+0,0.3-1.5);
\draw[thick] (0+0,0.6-1.5)--(0+0,0.9-1.5);
\draw[thick] (0+0.3+0.3,0.3-1.5)--(0+0.3+0,0-1.5);
\draw[thick] (0+0.3+0.2,0.1-1.5)--(0+0.3+0.3,0-1.5);
\draw[thick] (0+0.3+0,0.3-1.5)--(0+0.3+0.1,0.2-1.5);
\draw[thick] (0+0.3,0.3+0.3-1.5)--(0+0,0.3+0-1.5);
\draw[thick] (0+0.2,0.3+0.1-1.5)--(0+0.3,0.3+0-1.5);
\draw[thick] (0+0,0.3+0.3-1.5)--(0+0.1,0.3+0.2-1.5);
\draw[thick] (0+0.3+0.3,0.6+0.3-1.5)--(0+0.3+0,0.6+0-1.5);
\draw[thick] (0+0.3+0.2,0.6+0.1-1.5)--(0+0.3+0.3,0.6+0-1.5);
\draw[thick] (0+0.3+0,0.6+0.3-1.5)--(0+0.3+0.1,0.6+0.2-1.5);
\draw[->]    (0.3+0.3,0.3)--(0.3+0.3,0);
\draw[->]    (0.2,0.1)--(0+0.3,0);
\draw[->]    (0.1,0.1)--(0,0);
\draw[->]    (0.3+0.2,0.1-1.5)  --(0.3+0.3,0-1.5);
\draw[->]    (0.3+0.1,0.1-1.5)--(0.3+0,0-1.5);
\draw[->]    (0.3-0.3,0.3-1.5)--(0.3-0.3,0-1.5);
\end{tikzpicture}
\caption{A shackle move, a crossing doubling, and an Artin move.}
\label{Transformations}
\end{figure}
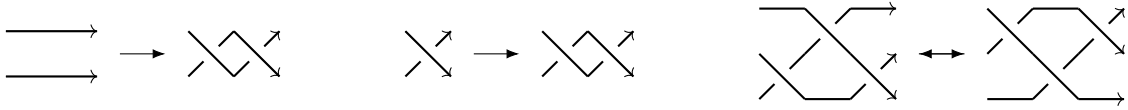

In contrast to the case of adequate diagrams, this combinatorial description does not allow us to check relatively quickly whether \eqref{R} is sharp. However, it makes it relatively easy to generate examples.

For positive diagrams of closed braids, Theorem \ref{MainTheorem} follows from \cite{GMM14}. Below we list some diagram types suggested by J. Gonz\`{a}lez-Meneses and P. M. G. Manch\`{o}n as well as new ones.

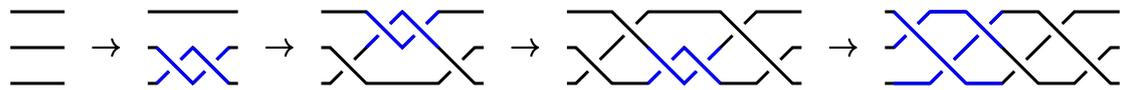
\begin{figure}[H]
\centering
\begin{tikzpicture}[rotate=90,scale=0.36575, every node/.style={scale=0.475}]
\pic[
  rotate=90,
  line width=1.15pt,
  braid/control factor=0,
  braid/nudge factor=0,
  braid/gap=0.11,
  braid/number of strands = 3,
  name prefix=braid,
] at (0,0) {braid={
1
}};
\draw[thick, -To] (1.3,-0.64-1.3-1) -- (1.3,-0.64-1.3-2); 
\pic[
  rotate=90,
  line width=1.15pt,
  braid/control factor=0,
  braid/nudge factor=0,
  braid/gap=0.11,
  braid/number of strands = 3,
  name prefix=braid,
] at (0,-0.64-1.3-3) {braid={
s_1^{-1}s_1^{-1}
}};
\draw[very thick, draw=blue] (1.3,-5.26) -- (0,-5.26-1.3); 
\draw[very thick, draw=blue] (0,-5.26) -- (0.44,-5.26-0.44); 
\draw[very thick, draw=blue] (2.16-1.3,-5.26-1.3+0.44) -- (2.6-1.3+0.01,-5.26-1.3-0.01); 
\draw[very thick, draw=blue] (1.3,-5.26-1.3) -- (0,-5.26-1.3-1.3); 
\draw[very thick, draw=blue] (0,-5.26-1.3) -- (0.44,-5.26-0.44-1.3); 
\draw[very thick, draw=blue] (2.16-1.3,-5.26-1.3+0.44-1.3) -- (2.6-1.3+0.01,-5.26-1.3-0.01-1.3); 
\draw[thick, -To] (1.3,-0.64-1.3-3 - 0.64-2.6-1) -- (1.3,-0.64-1.3-3 - 0.64-2.6-2); 
\pic[
  rotate=90,
  line width=1.15pt,
  braid/control factor=0,
  braid/nudge factor=0,
  braid/gap=0.11,
  braid/number of strands = 3,
  name prefix=braid,
] at (0,-0.64-1.3-3 - 0.64-2.6-3) {braid={
s_1^{-1}s_2^{-1}s_2^{-1}s_1^{-1}
}};
\draw[very thick, draw=blue] (1.3+1.3,-12.8) -- (0+1.3,-12.8-1.3) -- (0.44+1.3,-12.8-1.3-0.44); 
\draw[very thick, draw=blue] (2.16-1.3+1.3,-12.8-1.3-1.3+0.44) -- (2.6-1.3+0.01+1.3,-12.8-1.3-1.3-0.01); 
\draw[very thick, draw=blue] (0+1.3,-12.8) -- (0.44+1.3,-12.8-0.44); 
\draw[very thick, draw=blue] (1.3+1.3,-12.8-1.3) -- (0+1.3,-12.8-1.3-1.3); 
\draw[very thick, draw=blue] (2.16-1.3+1.3,-12.8-1.3+0.44) -- (2.6-1.3+0.01+1.3,-12.8-1.3-0.01); 
\draw[thick, -To] (1.3,-0.64-1.3-3 - 0.64-2.6-3 - 0.64-5.2-1) -- (1.3,-0.64-1.3-3 - 0.64-2.6-3 - 0.64-5.2-2); 
\pic[
  rotate=90,
  line width=1.15pt,
  braid/control factor=0,
  braid/nudge factor=0,
  braid/gap=0.11,
  braid/number of strands = 3,
  name prefix=braid,
] at (0,-0.64-1.3-3 - 0.64-2.6-3 - 0.64-5.2-3) {braid={
s_1^{-1}s_2^{-1}s_1^{-1}s_1^{-1}s_2^{-1}s_1^{-1}
}};
\draw[very thick, draw=blue] (1.3,-22.94) -- (0,-22.94-1.3) -- (0.44,-22.94-1.3-0.44); 
\draw[very thick, draw=blue] (2.16-1.3,-22.94-1.3-1.3+0.44) -- (2.6-1.3+0.01,-22.94-1.3-1.3-0.01); 
\draw[very thick, draw=blue] (0,-22.94) -- (0.44,-22.94-0.44); 
\draw[very thick, draw=blue] (1.3,-22.94-1.3) -- (0,-22.94-1.3-1.3); 
\draw[very thick, draw=blue] (2.16-1.3,-22.94-1.3+0.44) -- (2.6-1.3+0.01,-22.94-1.3-0.01); 
\draw[thick, -To] (1.3,-0.64-1.3-3 - 0.64-2.6-3 - 0.64-5.2-3 - 0.64-7.8-1) -- (1.3,-0.64-1.3-3 - 0.64-2.6-3 - 0.64-5.2-3 - 0.64-7.8-2); 
\pic[
  rotate=90,
  line width=1.15pt,
  braid/control factor=0,
  braid/nudge factor=0,
  braid/gap=0.11,
  braid/number of strands = 3,
  name prefix=braid,
] at (0,-0.64-1.3-3 - 0.64-2.6-3 - 0.64-5.2-3 - 0.64-7.8-3) {braid={
s_2^{-1}s_1^{-1}s_2^{-1}s_1^{-1}s_2^{-1}s_1^{-1}
}};
\draw[very thick, draw=blue] (1.3+1.3,-31.78) -- (0+1.3,-31.78-1.3); 
\draw[very thick, draw=blue] (0+1.3,-31.78) -- (0.44+1.3,-31.78-0.44); 
\draw[very thick, draw=blue] (2.16-1.3+1.3,-31.78-1.3+0.44) -- (2.6-1.3+0.01+1.3,-31.78-1.3-0.01); 
\draw[very thick, draw=blue] (1.3,-31.78-1.3) -- (0,-31.78-1.3-1.3); 
\draw[very thick, draw=blue] (0,-31.78-1.3) -- (0.44,-31.78-0.44-1.3); 
\draw[very thick, draw=blue] (2.16-1.3,-31.78-1.3+0.44-1.3) -- (2.6-1.3+0.01,-31.78-1.3-0.01-1.3); 
\draw[very thick, draw=blue] (1.3+1.3,-31.78-2.6) -- (0+1.3,-31.78-1.3-2.6); 
\draw[very thick, draw=blue] (0+1.3,-31.78-2.6) -- (0.44+1.3,-31.78-0.44-2.6); 
\draw[very thick, draw=blue] (2.16-1.3+1.3,-31.78-1.3+0.44-2.6) -- (2.6-1.3+0.01+1.3,-31.78-1.3-0.01-2.6); 
\draw[very thick, draw=blue] (0,-31.78) -- (0,-31.78-1.3);
\draw[very thick, draw=blue] (2.6,-31.78-1.3) -- (2.6,-31.78-2.6);
\draw[very thick, draw=blue] (0,-31.78-2.6) -- (0,-31.78-3.9);
\end{tikzpicture}
\caption{The process of obtaining the positive full twist with three strands.}
\label{ObtainingFullTwistBraid}
\end{figure}

First, according to Fig.~\ref{ObtainingFullTwistBraid}, the full-twist braid diagrams satisfy Theorem~\ref{MainTheorem}. 

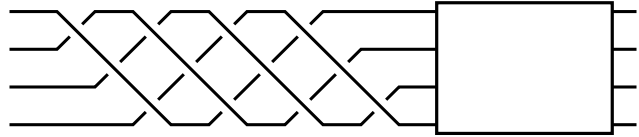
\begin{figure}[H]
\centering
\begin{tikzpicture}[rotate=90,scale=0.385, every node/.style={scale=0.5}]
\pic[
  rotate=90,
  line width=1.15pt,
  braid/control factor=0,
  braid/nudge factor=0,
  braid/gap=0.11,
  braid/number of strands = 4,
  name prefix=braid,
] at (0,0) {braid={
1
s_3^{-1}
s_2^{-1}
s_1^{-1}-s_3^{-1}
s_2^{-1}
s_1^{-1}-s_3^{-1}
s_2^{-1}
s_1^{-1}-s_3^{-1}
s_2^{-1}
s_1^{-1}
111111
}};
\draw[draw=black, very thick, fill=white] (-0.3,-0.32-13-1-0.28) rectangle ++(3.9+0.6,-6);
\end{tikzpicture}
\caption{A positive braid with the full twist.}
\label{FullTwistBraid}
\end{figure}

Second, 
one can obtain an arbitrary positive braid containing the full twist (see Fig. \ref{FullTwistBraid})
by doubling appropriate crossings on the bottom row of the full twist and applying suitable Artin moves as shown in Fig. \ref{ObtainingBraisWithFullTwist}.

\begin{figure}[H]
\centering
\begin{tikzpicture}[rotate=90,scale=0.385, every node/.style={scale=0.5}]
\pic[
  rotate=90,
  line width=1.15pt,
  braid/control factor=0,
  braid/nudge factor=0,
  braid/gap=0.11,
  braid/number of strands = 4,
  name prefix=braid,
] at (0,0) {braid={
s_3^{-1}
s_2^{-1}
s_1^{-1}
s_3^{-1}
s_2^{-1}
s_1^{-1}
s_3^{-1}
s_2^{-1}
s_1^{-1}
s_1^{-1}-s_3^{-1}
s_2^{-1}
s_1^{-1}
}};
\draw[very thick, draw=blue] (1.3,-10.72) -- (0,-10.72-1.3); 
\draw[very thick, draw=blue] (0,-10.72) -- (0.44,-10.72-0.44); 
\draw[very thick, draw=blue] (2.16-1.3,-10.72-1.3+0.44) -- (2.6-1.3+0.01,-10.72-1.3-0.01); 
\draw[very thick, draw=blue] (1.3,-10.72-1.3) -- (0,-10.72-1.3-1.3); 
\draw[very thick, draw=blue] (0,-10.72-1.3) -- (0.44,-10.72-0.44-1.3); 
\draw[very thick, draw=blue] (2.16-1.3,-10.72-1.3+0.44-1.3) -- (2.6-1.3+0.01,-10.72-1.3-0.01-1.3); 
\draw[thick, -To] (1.3+0.65,-0.64-15.6-1) -- (1.3+0.65,-0.64-15.6-2); 
\pic[
  rotate=90,
  line width=1.15pt,
  braid/control factor=0,
  braid/nudge factor=0,
  braid/gap=0.11,
  braid/number of strands = 4,
  name prefix=braid,
] at (0,-0.64-15.6-3) {braid={
s_3^{-1}
s_2^{-1}
s_1^{-1}
s_3^{-1}
s_2^{-1}
s_1^{-1}
s_3^{-1}
s_2^{-1}
s_1^{-1}
s_3^{-1}
s_2^{-1}
s_1^{-1}
s_2^{-1}
}};
\draw[very thick, draw=blue] (1.3+1.3,-32.56) -- (0+1.3,-32.56-1.3); 
\draw[very thick, draw=blue] (0+1.3,-32.56) -- (0.44+1.3,-32.56-0.44); 
\draw[very thick, draw=blue] (2.16-1.3+1.3,-32.56-1.3+0.44) -- (2.6-1.3+0.01+1.3,-32.56-1.3-0.01); 
\draw[very thick, draw=blue] (1.3,-32.56-1.3) -- (0,-32.56-1.3-1.3); 
\draw[very thick, draw=blue] (0,-32.56-1.3) -- (0.44,-32.56-0.44-1.3); 
\draw[very thick, draw=blue] (2.16-1.3,-32.56-1.3+0.44-1.3) -- (2.6-1.3+0.01,-32.56-1.3-0.01-1.3); 
\draw[very thick, draw=blue] (1.3+1.3,-32.56-2.6) -- (0+1.3,-32.56-1.3-2.6); 
\draw[very thick, draw=blue] (0+1.3,-32.56-2.6) -- (0.44+1.3,-32.56-0.44-2.6); 
\draw[very thick, draw=blue] (2.16-1.3+1.3,-32.56-1.3+0.44-2.6) -- (2.6-1.3+0.01+1.3,-32.56-1.3-0.01-2.6); 
\draw[very thick, draw=blue] (0,-32.56) -- (0,-32.56-1.3);
\draw[very thick, draw=blue] (2.6,-32.56-1.3) -- (2.6,-32.56-2.6);
\draw[very thick, draw=blue] (0,-32.56-2.6) -- (0,-32.56-3.9);
\end{tikzpicture}
\caption{The process of obtaining positive braids with the full twist.}
\label{ObtainingBraisWithFullTwist}
\end{figure}
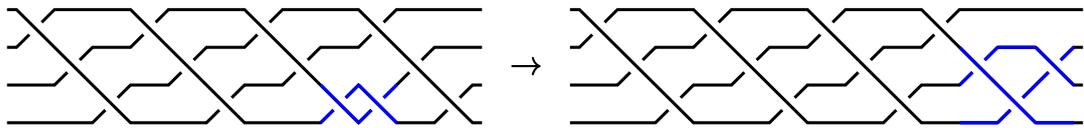

Equally, as shown in Fig. \ref{PositiveKnittedStencil}, similar positive diagrams with arbitrary configurations of Seifert circles also satisfy Theorem~\ref{MainTheorem}. Besides, see Fig.~\ref{ShackeExample} for a more complicated cable-like example.

In some cases, one can simplify the condition of Theorem~\ref{MainTheorem}. The example is given by positive diagrams without nested\footnote{Meaning that either the inside or the outside of each Seifert circle contains no other Seifert circles.} Seifert circles. Note, however, that any such diagram is alternating.

\begin{proposition}\label{NoNested}
If a positive diagram without nested Seifert circles contains no lone\footnote{A crossing is {\it lone} if it is the unique crossing joining the corresponding pair of Seifert circles.} crossings, then it can be obtained from a zero-crossing diagram by shackle moves and crossing doublings.
\end{proposition}
In fact, the sharpness of~\eqref{R} for these diagrams follows from the following criterion of Y. Diao, G. Hetyei, and P. Liu: for an alternating diagram $\D$, the bounds \eqref{L} and~\eqref{R} are sharp if and only if $\D$ contains no lone crossings \cite{DHL19}.

\begin{figure}
\centering
\includegraphics[width = 10cm]{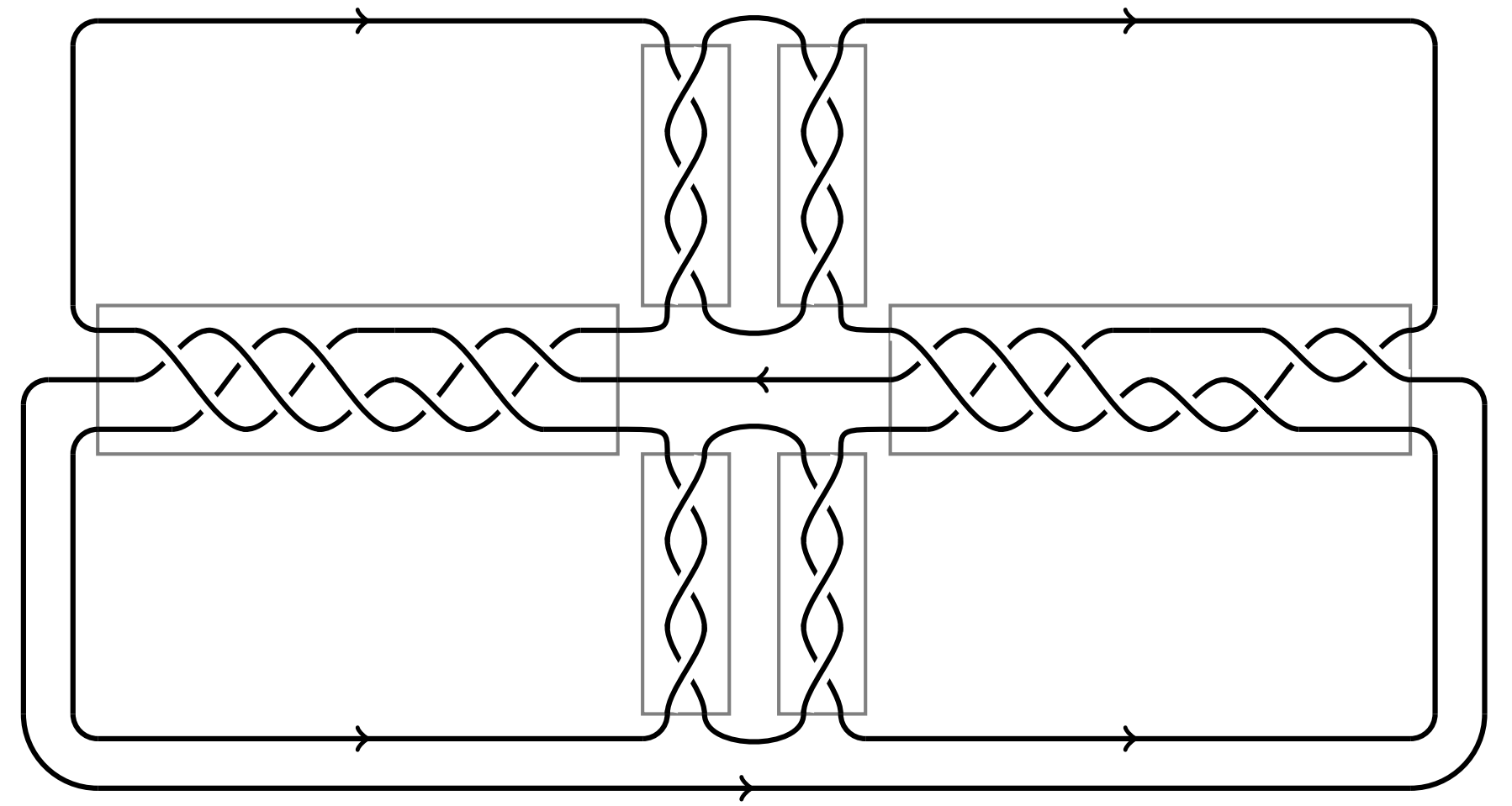}
\captionsource{A positive diagram with full twists.}{Image credit: \cite{Nak21}.}
\label{PositiveKnittedStencil}
\end{figure}

\subsection{Complexity of links}

Here we explicitly state inequalities and minimality results announced in the introduction.

Recall that there is a standard way to associate to a diagram~$\D$ of an oriented link $\L$ an oriented surface $\Sigma_\D$ spanning $\L$, which is called the {\it canonical Seifert surface} of $\D$. This surface is homotopy equivalent to its spine, which is a graph $\Gamma(\D)$, called the {\it Seifert graph} of~$\D$, whose vertices are identified with Seifert circles of $\D$ and whose edges are identified with crossings of $\D$. The Euler characteristic of the surface $\Sigma_\D$ equals $s(\D) - {\rm Cr}(\D)$, thus, its genus satisfies~$2g(\Sigma_\D) = 2 - \#\L + {\rm Cr}(\D) - s(\D),$ where $\#\L$ denotes the number of link components of $\L$. In this terms, the inequality \eqref{U} means ${\rm deg}_z^+\P(\D) - \#\L + 1 \leqslant 2g(\Sigma_\D)$, thus, it provides a bound for the {\it canonical genus}, i.e.~the minimum genus of canonical surfaces among all diagrams of a link. 

We highlight the following classical estimate (see \cite{BZ03}):
\begin{align}\label{ConwayGenusInequality}
\tag{U$^\prime$}
{\rm deg}_z^+\nabla(\D) - \#\L + 1 \leqslant 2g(\L),
\end{align}
where $g(\L)$ is the {\it Seifert genus}, i.e.~the minimum genus among {\it all} oriented surfaces spanning a link~$\L$, and $\nabla$ is Conway's version of the Alexander-Conway polynomial obtained from~$\P$ by the substitution $a = 1$. By combining \eqref{U} and \eqref{ConwayGenusInequality}, we have
\begin{align*}
{\rm deg}_z^+\nabla(\D) - \#\L+1 \leqslant \begin{matrix} {\rm deg}_z^+\mathcal{P}(\D) - \#\L+1, \\ 2g(\L) \end{matrix} \leqslant 2g(\Sigma_\D).
\end{align*}
In fact, for positive diagrams, ${\rm deg}_z^+\nabla(\D) - \#\L+1 = g(\Sigma_\D)$ (see \cite{Cro89}).

Recall that the quantity ${\rm sl}(\D) := \omega(\D) - s(\D)$ is called the {\it self-linking number} of a diagram~$\D$. In this terms, the inequality \eqref{L} means ${\rm sl}(\D) \leqslant {\rm deg}_a^-\P(\D) - 1$, thus, it provides a bound for the {\it maximal self-linking number}, i.e.~the maximum self-linking number among all diagrams of a link.

Denote by $\overline{\D}$ the {\it mirror image} of $\D$, i.e.~the diagram obtained from $\D$ by reversing the types of its crossings. By \cite{LM87}, the HOMFLY-PT polynomial of $\overline{\D}$ is obtained from~$\P(\D)$ by the substitution $a\mapsto -a^{-1}$. In particular, the inequality \eqref{R} provides a bound for the maximal self-linking number as well: ${\rm sl}(\overline{\D}) \leqslant {\rm deg}_a^-\P(\overline{\D}) - 1$.

\begin{corollary}
Suppose a positive diagram $\D$ can be obtained from a zero-crossing diagram by shackle moves, crossing doublings, and Artin moves. Then the self-linking number of the negative\footnote{An oriented diagram is {\it negative} if any of its crossings is negative.} diagram $\overline{\D}$ realizes the maximal self-linking number of the corresponding link.
\end{corollary}

The estimates \eqref{L} and \eqref{R} imply the following one, usually referred to as the {\it Morton-Franks-Williams inequality}:
\begin{align}\label{MFWInequality}
\tag{LR}
\dfrac{1}{2} \left({\rm deg}_a^+\P(\D) - {\rm deg}_a^-\P(\D)\right) + 1 \leqslant s(\D).
\end{align}
It provides a bound for the minimum Seifert circle number of the corresponding link. By~\cite{Yam87}, such a minimum equals the {\it braid index}, i.e.~the minimum number of strands among all closed braid representatives of a link.

\begin{corollary}
Suppose a positive diagram $\D$ can be obtained from a zero-crossing diagram by shackle moves, crossing doublings, and Artin moves. Then the number of Seifert circles of $\D$ realizes the braid index of the corresponding link.
\end{corollary}

Finally, we state the crossing number minimality result.

\begin{corollary}
Suppose a positive diagram $\D$ can be obtained from a zero-crossing diagram by shackle moves, crossing doublings, and Artin moves. Then the number of crossings of $\D$ realizes the crossing number of the corresponding link.
\end{corollary}

In general, as Fig. \ref{PositiveBraidReduction} shows, the number of crossings and the number of Seifert circles of a positive diagram without isthmuses may not be minimal.

\begin{figure}[H]
\centering
\begin{tikzpicture}[rotate=90,scale=0.385, every node/.style={scale=0.5}]
\pic[
  rotate=90,
  line width=1.15pt,
  braid/control factor=0,
  braid/nudge factor=0,
  braid/gap=0.11,
  braid/number of strands = 3,
  name prefix=braid,
] at (0,0) {braid={
s_2^{-1}s_1^{-1}s_2^{-1}s_1^{-1}s_2^{-1}
}};
\draw[very thick, draw=blue] (2.6+0.01,-1.62-1.3+0.01) -- (0,-1.62-1.3-1.3-1.3) -- (0, -7.13); 
\draw[very thick, draw=blue] (0,0) -- (0,-1.62) -- (0.44,-1.62-0.44); 
\draw[very thick, draw=blue] (0.85,-2.48) -- (1.75,-3.38); 
\draw[very thick, draw=blue] (2.16,-1.62-1.3-1.3+0.44) -- (2.6+0.01, -1.62-1.3-1.3-0.01); 
\draw[very thick, draw=blue] (0,0) .. controls ++(0,1.5) and ++(0,1.5) .. (-2,0)
    -- (-2,-7.13) .. controls ++(0,-1.5) and ++(0,-1.5) .. (0,-7.13);
\draw [thick,-Straight Barb, draw=blue] (-2,-3.3) -- (-2,-3.2);
\draw[very thick] (1.3,0) .. controls ++(0,3) and ++(0,3) .. (-3.3,0) 
-- (-3.3,-7.13) .. controls ++(0,-3) and ++(0,-3) .. (1.3,-7.13);
\draw [thick,-Straight Barb] (-3.3,-3.3) -- (-3.3,-3.2);
\draw[very thick] (2.6,0) .. controls ++(0,4.5) and ++(0,4.5) .. (-4.6,0)
    -- (-4.6,-7.13) .. controls ++(0,-4.5) and ++(0,-4.5) .. (2.6,-7.13);
\draw [thick,-Straight Barb] (-4.6,-3.3) -- (-4.6,-3.2);
\draw[thick, -To] (-1,-11.7) -- (-1,-12.7);
\pic[
  rotate=90,
  line width=1.15pt,
  braid/control factor=0,
  braid/nudge factor=0,
  braid/gap=0.11,
  braid/number of strands = 2,
  name prefix=braid,
] at (0,-14.63-1.5) {braid={
s_1^{-1}s_1^{-1}s_1^{-1}s_1^{-1}
}};
\draw[very thick, draw=blue] (1.3,-14.63-1.5-0.32-1.3) -- (0,-14.63-1.5-0.32-1.3-1.3) -- (0.44,-14.63-1.5-0.32-1.3-1.3-0.44); 
\draw[very thick, draw=blue] (2.16-1.3,-14.63-1.5-0.32-1.3-1.3-1.3+0.44) -- (2.6-1.3+0.01,-14.63-1.5-0.32-1.3-1.3-1.3-0.01); 
\draw[very thick ] (0,-14.63-1.5) .. controls ++(0,1.5) and ++(0,1.5) .. (-2,-14.63-1.5)
    -- (-2,-14.63-1.5-5.84) .. controls ++(0,-1.5) and ++(0,-1.5) .. (0,-14.63-1.5-5.84);
\draw [thick,-Straight Barb] (-2,-14.63-1.5-3.3) -- (-2,-14.63-1.5-3.2);
\draw[very thick] (1.3,-14.63-1.5) .. controls ++(0,3) and ++(0,3) .. (-3.3,-14.63-1.5) 
-- (-3.3,-14.63-1.5-5.84) .. controls ++(0,-3) and ++(0,-3) .. (1.3,-14.63-1.5-5.84);
\draw [thick,-Straight Barb] (-3.3,-14.63-1.5-3.3) -- (-3.3,-14.63-1.5-3.2);
\end{tikzpicture}
\caption{A reduction of a closed positive braid.}
\label{PositiveBraidReduction}
\end{figure}

\begin{figure}
\centering
\begin{tikzpicture}[rotate=90,scale=0.385, every node/.style={scale=0.5}]
\pic[
  rotate=-90,
  line width=1.15pt,
  braid/control factor=0,
  braid/nudge factor=0,
  braid/gap=0.11,
  braid/number of strands = 6,
  name prefix=braid,
] at (0,0) {braid={
s_4^{-1}
s_3^{-1}-s_5^{-1}
s_4^{-1}
s_1^{-1}-s_4^{-1}
s_1^{-1}-s_3^{-1}-s_5^{-1}
s_1^{-1}-s_4^{-1}
s_4^{-1}
s_3^{-1}-s_5^{-1}
s_2^{-1}-s_4^{-1}
s_1^{-1}-s_3^{-1}
s_2^{-1}s_2^{-1}
s_1^{-1}-s_3^{-1}
s_2^{-1}s_2^{-1}
s_1^{-1}-s_3^{-1}
s_2^{-1}
}};
\end{tikzpicture}
\caption{A positive braid obtained from a trivial one by shackle moves and crossing doublings.}
\label{ShackeExample}
\end{figure}

\subsection{Strongly quasipositive links}%

Here we suggest a possible generalization of Theorem~\ref{MainTheorem} from positive links to strongly quasipositive ones. Recall that strongly quasipositive braids are products of positive Birman-Ko-Lee band generators, and strongly quasipositive links are closures of strongly quasipositive braids. 

By \cite{Rud05}, the canonical Seifert surface of any (positive) link diagram is isotopic to the so-called Bennequin surface of some closed (strongly quasipositive) braid. Such an isotopy preserves, in some sense, the local structure of the diagrams. In particular, if a positive link diagram $\D$ can be obtained from a zero-crossing diagram by shackle moves, crossing doublings, and Artin moves, then the corresponding strongly quasipositive braid can be obtained from the trivial braid with $s(\D)$ strands by adding squares of positive band generators, doublings of band generators, and Birman-Ko-Lee moves\footnote{That is, the standard relations between the band generators in the strongly quasipositive braid monoids.}. The resulting strongly quasipositive braids satisfy the equalities in~\eqref{ConwayGenusInequality}, \eqref{L}, and~\eqref{R}. It may turn out that Theorem~\ref{MainTheorem} is just a special case of some more general result concerning the Alexander-Conway and the HOMFLY-PT polynomials of strongly quasipositive braids.

\begin{conjecture}\label{Conjecture}
For the closure of any strongly quasipositive braid obtained from a trivial braid by adding squares of positive band generators, doublings of band generators, and Birman-Ko-Lee moves,
\begin{enumerate}
\item[(i)] the bound~\eqref{ConwayGenusInequality} is sharp;
\item[(ii)] the bound~\eqref{L} is sharp;
\item[(iii)] the bound~\eqref{R} is sharp.
\end{enumerate}
\end{conjecture}

Note that in general the inequality \eqref{ConwayGenusInequality} may be strict\footnote{For example, the positive Whitehead double of the trefoil is a non-trivial strongly quasipositive knot. However, the Alexander-Conway polynomial of any Whitehead double is trivial.} for strongly quasipositive links. Besides, \eqref{L} may be strict as well (see \cite[Example~4.2]{HS15} and \cite[Example~3]{St05}).

An interesting class of strongly quasipositive braids satisfying requirements of Conjecture~\ref{Conjecture} is of those containing the square of the so-called dual Garside element. By \cite{Ban22}, any such strongly quasipositive braid can be obtained from such a square by doublings of band generators and Birman-Ko-Lee moves (cf. Fig \ref{ObtainingBraisWithFullTwist}). 
According to I. Banfield, since at the level of Bennequin surfaces, doublings of band generators may be realized as plumbings of positive Hopf bands and Birman-Ko-Lee moves may be realized as isotopy, the closures of such braids are fibered.

There is an explanation of the sharpness of the bound~\eqref{L} for positive diagrams in terms of contact knot theory (see \cite{Etn05}). According to \cite{HS15}, if a link admits a Legendrian representative spanning a Lagrangian filling, i.e.~a properly embedded Lagrangian submanifold of the four-ball, then the bound is sharp. Besides, the Legendrian representatives of positive diagrams are fillable. It would be interesting to find out whether links satisfying requirements of Conjecture \ref{Conjecture} admit fillable Legendrian representatives.

As for the other part of Theorem~\ref{MainTheorem}, we do anticipate a similar criterion for the sharpness of \eqref{R} for strongly quasipositive links but expect that it might be more complex.

The paper is organized as follows. 
In Section 2, we recall the construction of castles from \cite{DHL19} and prove Proposition \ref{NoNested}. Section 3 aims to adapt resolution trees for the HOMFLY-PT polynomial introduced in \cite{DHL19}. In Section 4, we prove Theorem \ref{MainTheorem}.

\subsection*{Acknowledgements}

This paper is a revised version of the author's Master's Thesis.
The author is deeply grateful to his advisor Andrei Malyutin for his inspiration and support.
Besides, the author would like to thank the participants of the Low-dimensional topology student seminar of the Leonhard Euler International Mathematical Institute in Saint Petersburg for their attention to this work.
Lastly, the author thanks Danil Nigomedyanov for a thorough reading of the Master’s Thesis.

\section{Link diagram scrutiny}

A {\it link} $\L$ is a smooth one-dimensional submanifold of the euclidean space $\mathbb{R}^3$. Connected components of $\L$ are called its {\it link components}. A projection of $\L$ onto a plane is {\it regular} if it has only finitely many singular points that are all transverse double points. A {\it diagram}~$\D$ of~$\L$ is its regular projection that has relative height information added to it at each of the double points. Namely, a break in the line corresponding to the thread that passes underneath. A {\it crossing} of $\D$ is a broken double point. If $\L$ is oriented, one can naturally induce an orientation to $\D$.

We embed each projection plane into its one-point compactification (homeomorphic to a two-sphere), which we denote by $\S$, and endow it with an orientation. In particular, each Seifert circle cuts $\S$ into two closed disks. Besides, if the link diagram is oriented, by comparing the orientation of a Seifert circle induced from it with the orientation of $\S$, we can distinguish the left disk from the right one. A Seifert circle is {\it loose on the left} (resp. {\it loose on the right}) if its spanning disk that lies on the left (resp. on the right) contains no other Seifert circles. Finally, a Seifert circle is {\it innermost} if at least one of its spanning disks contains no other Seifert circles.

For our purposes, we consider link diagrams up to orientation-preserving homeomorphisms of the two-sphere $\S$.

\subsection{Descents and ascents}

By a {\it point on a diagram} $\D$ we mean a point lying on a Seifert circle of $\D$. Let $x_1,x_2,\ldots,x_m$ be a sequence of points on~$\D$, one on each link component of~$\D$, where~$m := \#\D$. We refer to~$x_1,x_2,\ldots,x_m$ as {\it base points}.

We travel through $\D$ by moving along each link component of $\D$ as follows. We begin at the first base point $x_1$ and move according to the orientation. As we reach $x_1$ again, we proceed to the second marked point $x_2$ and begin moving according to the orientation. We continue in the same way until we visit the link component containing $x_m$ entirely.

During the above travel, we visit each crossing of $\D$ exactly twice. A crossing of $\D$ is {\it descending} (resp.\  {\it ascending}) if one travels along the overpassing (resp.\  underpassing) strand first. The diagram~$\D$ is {\it descending} (resp.\  {\it ascending}) if each crossing of~$\D$ is descending (resp.\  ascending). Of course, these properties depend on the choice of the base points.

\begin{figure}[H]
\centering
\begin{tikzpicture}[rotate=90,scale=0.385, every node/.style={scale=0.5}]
\pic[
  rotate=90,
  line width=1.15pt,
  braid/control factor=0,
  braid/nudge factor=0,
  braid/gap=0.11,
  braid/number of strands = 3,
  braid/strand 1/.style={blue},
  name prefix=braid,
] at (0,0) {braid={
s_1s_2s_2^{-1}s_2s_1^{-1}s_2^{-1}s_2s_1^{-1}s_1^{-1}
}};
\draw[very thick, draw=blue] (2.6,0) -- (2.6,-1.62);
\draw[very thick, draw=blue] (2.16,-2.06) -- (2.61,-1.61);
\draw[very thick, draw=blue] (0.83+0.47+0.47,-3.39) -- (0.83+0.47,-3.39+0.47) -- (1.73,-2.49);
\draw[very thick, draw=blue] (2.6+0.01,-2.92-1.3-0.01) -- (2.17,-2.49-1.3);
\draw[very thick, draw=blue] (2.16,-2.06-2.6) -- (2.61,-1.61-2.6);
\draw[very thick, draw=blue] (0.47,-3.39-2.6-0.83-2.6-0.47) -- (0,-3.39-2.6-0.83-2.6) -- (0.83-0.83,-3.39-2.6-0.83) -- (1.73,-2.49-2.6);
\draw[very thick] (-2,0) -- (-2,-4.5-7.8);
\draw[very thick] (-3.3,0) -- (-3.3,-4.5-7.8);
\draw[very thick] (-4.6,0) -- (-4.6,-4.5-7.8);
\draw[very thick] (0,0) .. controls ++(0,1.5) and ++(0,1.5) .. (-2,0)
    -- (-2,-4.5-7.8) .. controls ++(0,-1.5) and ++(0,-1.5) .. (0,-4.5-7.8);
\draw [thick,-Straight Barb] (-2,-2.35-3.9) -- (-2,-2.25-3.9);
\draw[very thick ] (1.3,0) .. controls ++(0,3) and ++(0,3) .. (-3.3,0)
    -- (-3.3,-4.5-7.8) .. controls ++(0,-3) and ++(0,-3) .. (1.3,-4.5-7.8);
\draw [thick,-Straight Barb] (-3.3,-2.35-3.9) -- (-3.3,-2.25-3.9);
\draw[very thick, draw=blue] (2.6,0) .. controls ++(0,4.5) and ++(0,4.5) .. (-4.6,0)
    -- (-4.6,-4.5-7.8) .. controls ++(0,-4.5) and ++(0,-4.5) .. (2.6,-4.5-7.8);
\draw [draw=blue,thick,-Straight Barb] (-4.6,-2.35-3.9) -- (-4.6,-2.25-3.9);
\filldraw[blue] (0,0) circle (3pt) {};
\node[scale=2] at (0.5,0) {$p$};
\end{tikzpicture}
\caption{A maximal descending path.}
\label{StandardBasePointsExample}
\end{figure}

It is easy to see that if $\D$ is descending (resp.\  ascending), then its components represent the unknots and are layered from top to bottom (resp.\  from bottom to top). In this case, $\D$ represents the unlink\footnote{A link is called an {\it unlink} if it admits a diagram with zero crossings.} with precisely $m$ link components.

Next, let $\D$ be an arbitrary link diagram, and suppose we travel through $\D$ as above. Given a point $p$ on $\D$, let the {\it maximal descending path} be the longest path starting at~$p$ that we traveled before meeting either~$p$ or an ascending crossing (see Fig. \ref{StandardBasePointsExample}). We define the {\it maximal ascending path} similarly. 

\subsection{Castles}

We introduce castles, which are additional structures for link diagrams, by adapting the construction of Y. Diao, G. Hetyei, and P. Liu \cite{DHL19}.

For this section, there is no necessity to specify crossing types. In particular, we interpret a crossing as an arc joining two Seifert circles.

\subsubsection{The construction}

A castle consists of several segments on Seifert circles, called floors, and several crossings between them, called ladders. Each floor is endowed with a non-negative integer referred to as its {\it level}. According to the definition, each castle has a unique floor of level zero and may have floors of higher levels.

We say that a point $x$ on $\D$ is {\it innermost} if the Seifert circle of $\D$ containing $x$ is innermost.

Let $x$ be an innermost point on~$\D$. We describe an algorithm for constructing the castle~${\rm Cas}(x;\D)$ of~$\D$ determined by~$x$. The construction is inductive, so we build each floor step by step.

At first, we construct the unique floor of level $0$. Denote by $C$ the innermost Seifert circle of $\D$ containing $x$. Starting at~$x$ and following the orientation of $C$, let us order all crossings of $\D$ incident to~$C$. According to the definition, the {\it floor of level~$0$} is the segment on $C$ whose boundary points are~$p_0:=x$ and the point~$q_0$ located immediately after the last crossing incident to~$C$. 

To construct floors of level $k \in \{1,2,\ldots\}$, we proceed as follows.

We assume that all floors of level $k-1$ have already been built. Let $F$ be any of them and denote by $p_{k-1}$ and $q_{k-1}$ the starting and the ending points of $F$, respectively. We construct all floors of level~$k$ incident to $F$ as follows. 

Let $C^\prime$ be a Seifert circle of $\D$ that shares at least one crossing with $F$ and does not contain floors of level $k-2$. Starting at~$p_{k-1}$ and following the orientation of $F$, let us order all of the mutual crossings. Let $p_k$ and $q_k$ be two points on~$C^\prime$ immediately before the first such crossing and immediately after the last one, respectively. The segment of~$C^\prime$ starting at $p_k$ and ending at~$q_k$ is a {\it floor of level~$k$}.

In the same way, we build floors of level~$k$ for all other Seifert circles of~$\D$ that share crossings with~$F$, which may lie on both sides of $F$. After that, we repeat the above construction for all floors of level $k-1$. This completes the definition of floors of level $k$. The construction of the floors continues until, for each floor $F$ of maximal level, any Seifert circle sharing crossings with $F$ contains a floor of level one less. 

\begin{figure}[H]
\centering
\includegraphics[width = 15.3cm]{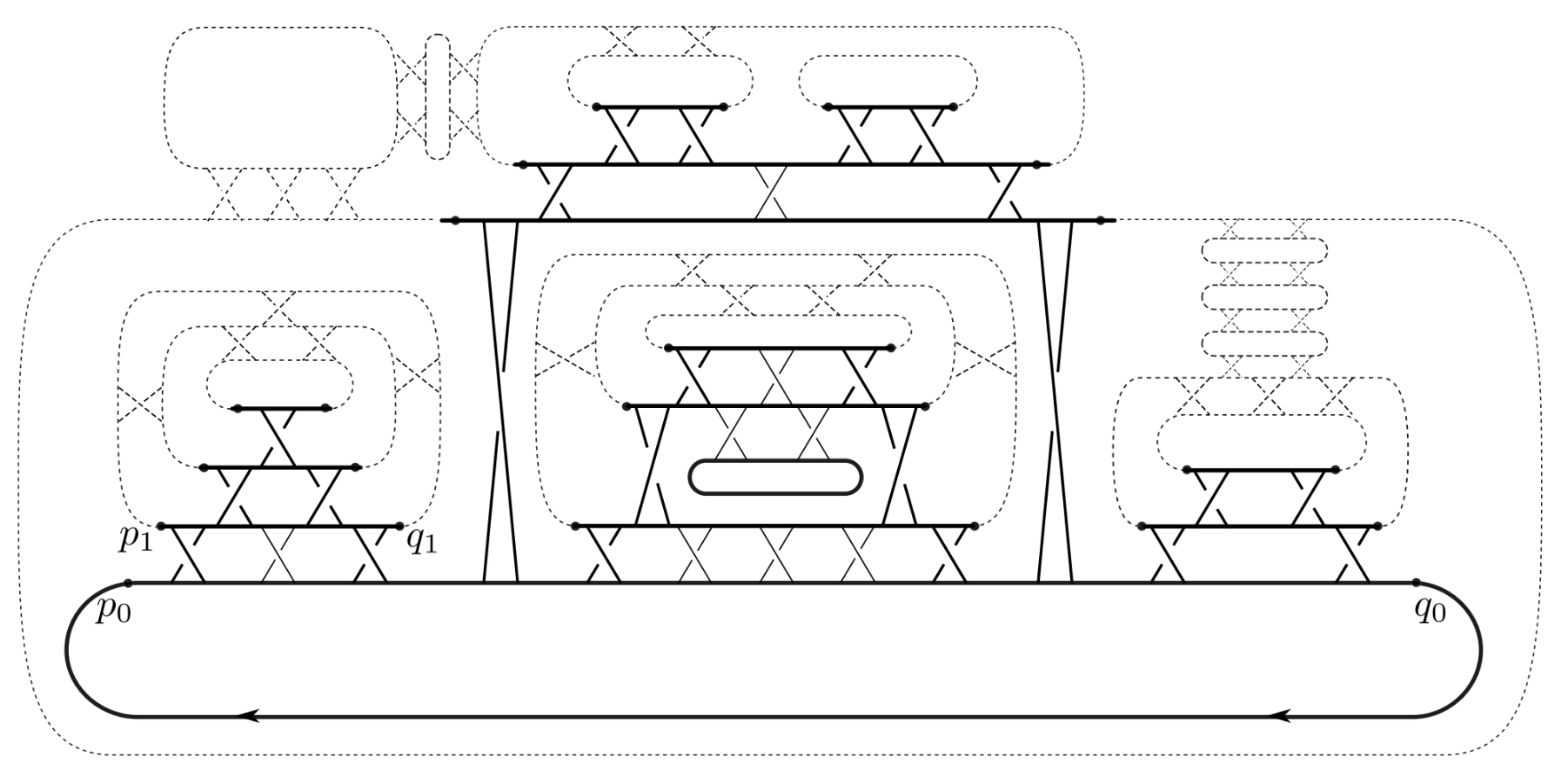}
\captionsource{A castle.}{Image credit: \cite{DHL19}.}
\label{Castle}
\end{figure}

A crossing $s$ of $\D$ is called a {\it ladder} if it joins two floors of the castle. According to the definition, the castle ${\rm Cas}(x;\D)$ is the collection of all such floors and ladders. This completes definition.

Finally, it is worth mentioning that there may be crossings of $\D$ that are not ladders of the castle and Seifert circles of $\D$ that do not contain floors of the castle. For instance, see dashed objects in~Fig.~\ref{Castle}. 

\subsubsection{Traps}

Given two floors of a castle, we order their common ladders according to the orientation. Two such distinct ladders are {\it adjacent} if they are consecutive in that order.

\begin{definition}\label{Braces}
Let $s_1$ and $s_2$ be adjacent ladders joining two floors $F_1$ and $F_2$. Denote by~$F_1^\prime$ and~$F_2^\prime$ the two segments of, respectively, $F_1$ and $F_2$ that join~$s_1$ and~$s_2$. The quadruple~$(s_1,s_2,F_1^\prime,F_2^\prime)$ is called a {\it brace}. The circle formed of~$s_1$,~$s_2$,~$F_1^\prime$, and~$F_2^\prime$ cuts the two-sphere into two closed disks. The disk that does not contain the base point $x$ is called the {\it inside} of the brace. A brace is called a {\it trap} if its inside contains at least one floor of the castle.
\end{definition}

For instance, any two ladders of the form 
\begin{tikzpicture}[rotate=90, scale=0.9, baseline=1]
\draw[thick] (0,0.3+0.3)--(0.3,0+0.3);
\draw[thick] (0.1,0.1+0.3)--(0,0+0.3);
\draw[thick] (0.3,0.3+0.3)--(0.2,0.2+0.3);
\draw[->] (0,0.3)--(0.3,0);
\draw[->] (0.1,0.1)--(0,0);
\draw[thick] (0,0.3)--(0.3,0);
\draw[thick] (0.1,0.1)--(0,0);
\draw[thick] (0.3,0.3)--(0.2,0.2);
\end{tikzpicture}
determine a brace. However, in general, there may be some ladders attached to the segments $F_1^\prime$ and $F_2^\prime$ of a brace. The brace is not a trap if and only if the only such ladders are outside of the brace. Therefore, the local structure of a castle without traps is similar to that of a braid.

See the blue brace in Fig.~\ref{TrappedCastleExample} for an example of a trap. Besides, the castle shown in Fig.~\ref{Castle} has two traps.

\begin{definition}\label{AppropriatePoints}
An innermost point $x$ is {\it appropriate} if ${\rm Cas}(x;\D)$ has no traps.
\end{definition}

The difference between our version of castles and the one of \cite{DHL19} is that we consider link diagrams on the two-sphere and take it into account when giving definitions. 
In \cite[Lemma~4.3]{DHL19}, Y. Diao, G. Hetyei, and P. Liu claim that any link diagram contains at least one appropriate point. However, there are flaws in their proof. In Lemma \ref{TrappedCastle} below, we correct it.

\begin{lemma}\label{TrappedCastle}
Any oriented link diagram contains an appropriate point.
\end{lemma}
\begin{proof}
Let $x$ be an innermost point on a diagram $\D$. We suppose $x$ is not appropriate and derive from $x$ other such points. After that, we show that any sequence of such derivations leads to an appropriate point.

Suppose the castle~${\rm Cas}(x;\D)$ has at least one trap. Let $(s_1,s_2,F_1^\prime,F_2^\prime)$ be the corresponding brace. By definition, its inside $\mathbb{D}$ contains at least one floor of the castle. Let~$C^\prime$ be a Seifert circle of $\D$ lying within $\mathbb{D}$ such that~$C^\prime$ contains a floor sharing crossings with either~$F_1$ or~$F_2$. The circle $C^\prime$ cuts the two-sphere into two closed regions. Let $C^{\prime\prime}$ be an arbitrary innermost Seifert circle of~$\D$ lying within the region that does not contain the base point $x$ (see Fig.~\ref{TrappedCastleExample}). It might be $C^{\prime\prime} = C^{\prime}$.

\begin{figure}
\centering
\begin{tikzpicture}[scale=0.8, every node/.style={scale=1}]
\draw[very thick, rounded corners=8pt, draw=blue] (0, 0) rectangle (18, 7) {}; 
\draw[draw=white, fill=white] (1.3,7-0.1) rectangle (16.7,7+0.1) {}; 
\draw[very thick, rounded corners=8pt, draw=blue] (1, 1) rectangle (17, 2.1) {};
\draw[draw=white, fill=white] (2.2,1-0.1) rectangle (15.8,1+0.1) {}; 
\draw[draw=black, very thick] (2.2,1) -- (15.8,1) {}; 
\draw[draw=white, fill=white] (2.2,0-0.1) rectangle (15.8,0+0.1) {}; 
\draw[draw=black, very thick] (2.2,0) -- (15.8,0) {}; 
\draw[fill=white, draw=white] (2.6, -0.1) rectangle (5.4, 1.1) {}; 
\draw[thin, draw=black, dashed] (2.5, 0) -- (5.5, 0) {};
\draw[thin, draw=black, dashed] (2.5, 1) -- (5.5, 1) {};
\draw[draw=black, very thick, draw=blue] (1.8, 0) -- ++(0.4, 1) -- ++(-0.4, 0) -- ++(0.4, -1) {};
\draw[draw=black, very thick] (5.8, 0) -- ++(0.4, 1) -- ++(-0.4, 0) -- ++(0.4, -1) {};
\draw[draw=black, very thick] (7.8, 0) -- ++(0.4, 1) -- ++(-0.4, 0) -- ++(0.4, -1) {};
\draw[draw=black, very thick] (9.8, 0) -- ++(0.4, 1) -- ++(-0.4, 0) -- ++(0.4, -1) {};
\draw[draw=black, very thick] (11.8, 0) -- ++(0.4, 1) -- ++(-0.4, 0) -- ++(0.4, -1) {};
\draw[draw=black, very thick] (13.8, 0) -- ++(0.4, 1) -- ++(-0.4, 0) -- ++(0.4, -1) {};
\draw[draw=black, very thick, draw=blue] (15.8, 0) -- ++(0.4, 1) -- ++(-0.4, 0) -- ++(0.4, -1) {};
\filldraw[black] (2.6,1) circle (2pt) {};
\filldraw[black] (2.6,0) circle (2pt) {};
\filldraw[black] (5.4,1) circle (2pt) node[anchor=south]{$x$};
\filldraw[black] (5.4,0) circle (2pt) {};
\draw[thin, rounded corners=8pt, draw=black, dashed] (1, 3) rectangle (11, 6) {};
\draw[draw=black, very thick] (1.3,6) -- (10.7,6) {}; 
\draw[draw=black, very thick] (1.5, 6) -- ++(0.4, 1) -- ++(-0.4, 0) -- ++(0.4, -1) {};
\draw[draw=black, very thick] (3.5, 6) -- ++(0.4, 1) -- ++(-0.4, 0) -- ++(0.4, -1) {};
\draw[draw=black, very thick] (6.5, 6) -- ++(0.4, 1) -- ++(-0.4, 0) -- ++(0.4, -1) {};
\draw[draw=black, very thick] (8.1, 6) -- ++(0.4, 1) -- ++(-0.4, 0) -- ++(0.4, -1) {};
\draw[draw=black, very thick] (10.1, 6) -- ++(0.4, 1) -- ++(-0.4, 0) -- ++(0.4, -1) {};
\filldraw[black] (1.3,6) circle (2pt) {};
\filldraw[black] (10.7,6) circle (2pt) {};
\draw[thin, rounded corners=8pt, draw=black, dashed] (2, 4) rectangle (8, 5) {};
\draw[draw=black, very thick] (2.5, 5) -- ++(0.4, 1) -- ++(-0.4, 0) -- ++(0.4, -1) {};
\draw[draw=black, very thick] (4.5, 5) -- ++(0.4, 1) -- ++(-0.4, 0) -- ++(0.4, -1) {};
\draw[draw=black, very thick] (7.1, 5) -- ++(0.4, 1) -- ++(-0.4, 0) -- ++(0.4, -1) {};
\filldraw[black] (2.3,5) circle (2pt) {};
\filldraw[black] (7.7,5) circle (2pt) node[anchor=north]{$y$};
\draw[draw=black, very thick] (2.3,5) -- (7.7,5) {}; 
\draw[thin, rounded corners=8pt, draw=black, dashed] (8.8, 4) rectangle (10.2, 5) {};
\draw[draw=black, very thick] (9.3, 5) -- ++(0.4, 1) -- ++(-0.4, 0) -- ++(0.4, -1) {};
\filldraw[black] (9.1,5) circle (2pt) {};
\filldraw[black] (9.9,5) circle (2pt) {};
\draw[draw=black, very thick] (9.1,5) -- (9.9,5) {}; 
\draw[thin, rounded corners=8pt, draw=black, dashed] (12, 3) rectangle (17, 6) {};
\draw[draw=black, very thick] (12.5, 6) -- ++(0.4, 1) -- ++(-0.4, 0) -- ++(0.4, -1) {};
\draw[draw=black, very thick] (14.3, 6) -- ++(0.4, 1) -- ++(-0.4, 0) -- ++(0.4, -1) {};
\draw[draw=black, very thick] (16.1, 6) -- ++(0.4, 1) -- ++(-0.4, 0) -- ++(0.4, -1) {};
\filldraw[black] (12.3,6) circle (2pt) {};
\filldraw[black] (16.7,6) circle (2pt) {};
\draw[draw=black, very thick] (12.3,6) -- (16.7,6) {}; 
\draw[draw=blue, very thick] (1.2,7) -- (16.8,7) {}; 
\node at (9,1.3) {$F_2$};
\node[text=blue] at (9,2.45) {$F_2^\prime$};
\node at (9,0.3) {$F_1$};
\node at (3,4.3) {$C^{\prime\prime}$};
\node at (2,3.3) {$C^{\prime}$}; 
\node[text=blue] at (17.6,4.3) {$F_1^\prime$}; 
\node[text=blue] at (1.6,0.5) {$s_2$};
\node[text=blue] at (16.4,0.5) {$s_1$};
\draw [-Straight Barb] (6.7,0) -- (6.8,0); 
\draw [-Straight Barb] (6.7,1) -- (6.8,1); 
\draw [blue, -{Straight Barb}] (0.8,7) -- (0.7,7); 
\draw [blue, -{Straight Barb}] (17.4,7) -- (17.3,7); 
\draw [-{Straight Barb}] (13.6,6) -- (13.5,6); 
\draw [-{Straight Barb}] (5.6,6) -- (5.5,6); 
\draw [-{Straight Barb}] (5.6,5) -- (5.5,5); 
\draw [blue, -{Straight Barb}] (6.9,2.1) -- (6.8,2.1); 
\end{tikzpicture}
\caption{A trap of a castle.}
\label{TrappedCastleExample}
\end{figure}
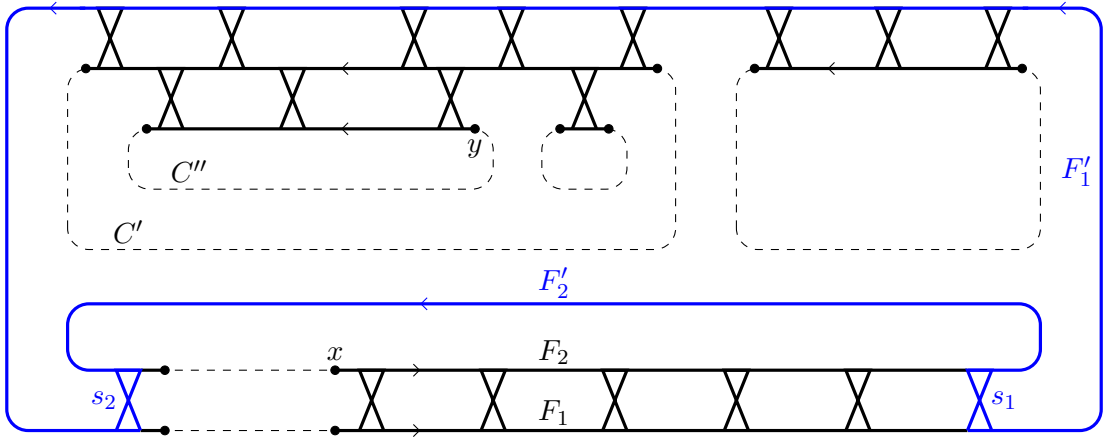

Without loss of generality, $C^\prime$ shares crossings with $F_1$. We say that a point~$y$ on~$C^{\prime\prime}$ is {\it derived from $x$} if any floor of ${\rm Cas}(y;\D)$ intersecting $F_1$ lies within $F_1^\prime$.

We claim that there exists at least one point derived from $x$.

Let us prove the claim. Let $y^\prime$ be an arbitrary point on $C^{\prime\prime}$. We suppose the castle~${\rm Cas}(y^\prime;\D)$ contains a floor intersecting~$F_1$, and we find below another suitable point.

Since $C^\prime$ separates $C^{\prime\prime}$ from $F_1$, the Seifert circle $C^\prime$ contributes a floor to the castle~${\rm Cas}(y^\prime;\D)$. Denote by $C_0, C_1, \ldots, C_m$ the Seifert circles of $\D$ forming the shortest path in the Seifert graph $\Gamma(\D)$ from~$C^{\prime\prime}$ to~$C^{\prime}$ such that~$C_0 = C^{\prime\prime}$, $C_m = C^\prime$, and each $C_i$ contributes a floor to ${\rm Cas}(y^\prime;\D)$. To derive at least one point $y$ from $x$, we proceed by induction on $m$.

At first, suppose $m=0$. In this case, $C^{\prime\prime} = C^\prime$. Following the orientation of~$F_1^\prime$, let us order all the crossings joining $C^\prime$ and $F_1^\prime$. Let $y$ be a point on $C^\prime$ immediately before the first such crossing. It is easy to see that $y$ fits the required condition.

Next, suppose $m \geqslant 1$. The Seifert circle $C_1$ cuts the two-sphere into two closed regions. First, let us smooth all crossings of $\D$ lying within the region that does not contain the base point $x$. Second, let us delete all Seifert circles lying within the interior of this region. Denote by $\D^\prime$ the resulting link diagram. By the induction hypothesis applied to $\D^\prime$, there is at least one point~$y^\prime$ on the Seifert circle $C_1$ of $\D^\prime$ such that $y^\prime$ is derived from $x$.

Starting at $y^\prime$ and following the orientation of~$C_1$, let us order all the crossings joining~$C_0$ and~$C_1$. We set~$y$ to be a point on~$C_0$ immediately before the first such crossing.

We are in the position of proving that $y$ is derived from~$x$. By the definition of $y$, any ladder of~${\rm Cas}(y^\prime;\D)$ that joins $C_1$ and $C_2$ is a ladder of ${\rm Cas}(y;\D)$. It follows by induction that for each index~$k \in \{2,3,\ldots,m+1\}$, if the castle~${\rm Cas}(y;\D)$ has a floor~$F$ corresponding to $C_k$, then the castle~${\rm Cas}(y^\prime;\D)$ has a floor containing $F$ (here $C_{m+1}$ stands for the Seifert circle containing $F_1$). In particular, the claim is proved.

It remains to show that by using the derivations described above, one obtains an appropriate pair. Since the Seifert graph $\Gamma(\D)$ is bipartite, $C^\prime$ and $F_2$ share no crossings. Thus, the castle~${\rm Cas}(y;\D)$ has no floors intersecting $F_2$. Therefore, the derivation decreases the number of floors in the complement of the interior of the current trap.
By the construction, any trap of ${\rm Cas}(y;\D)$ is either a trap of ${\rm Cas}(x;\D)$ or is contained within $\mathbb{D}$. Therefore, the new traps are contained in the current one, and thus, the derivations lead to a castle without traps. The lemma is proved.
\end{proof}

\subsubsection{Towers}

A sequence $F_0, F_1, \ldots, F_m$ of floors of the castle is called a {\it tower} if each $F_k$ is of level $k$, any two consecutive floors share at least one crossing, and $F_m$ is a floor of maximal level, i.e.~$F_m$ shares no crossings with floors of level $m+1$.

\begin{definition}\label{CoherentSequences}
Two Seifert circles are {\it coherent} if they are homologous in the annulus bounded by them. A sequence of Seifert circles is {\it coherent} if any two of its elements are coherent and any two consecutive ones share at least one crossing.
\end{definition}

In general, Seifert circles corresponding to floors lying in a tower may not be coherent. For example, the castle shown in Fig. \ref{Castle} has floors whose Seifert circles are incoherent with the Seifert circle corresponding to the zero level floor.

\begin{observation}\label{NoTrapsImpliesCoherence}
Suppose a castle has no traps. Then the Seifert circles corresponding to the floors in each of its towers form a coherent sequence.
\end{observation}

\subsection{Proof of Proposition \ref{NoNested}}

Let $\D$ be a positive diagram without nested Seifert circles. Assume that $\D$ has no lone crossings. To prove that $\D$ can be obtained from a zero-crossing diagram by shackle moves and crossing doublings, we proceed by induction on ${\rm Cr}(\D)$. If $\D$ has precisely zero crossings, there is nothing to prove. Next, suppose that ${\rm Cr}(\D) > 0$. 

Without loss of generality, $\D$ has no trivial components, i.e.~those components that are Seifert circles. By Lemma \ref{TrappedCastle}, there exists an appropriate point $x$ on $\D$. 
Since $\D$ has no trivial components, the castle~${\rm Cas}(x;\D)$ has at least one ladder. Denote by~$F_0$ the floor of level zero and let~$F_1$ be an arbitrary floor of level one. 

Since $\D$ has no nested Seifert circles, there are no floors on top of $F_1$. Besides, since the castle~${\rm Cas}(x;\D)$ has no traps, the ladders joining $F_0$ and $F_1$ are contained in a disk, whose intersection with $\D$ is a braid with two strands, such as
\begin{tikzpicture}[rotate=90, scale=0.9, baseline=1]
\draw[thick] (0.3,0.3+0.6)--(0,0+0.6);
\draw[thick] (0.2,0.1+0.6)--(0.3,0+0.6);
\draw[thick] (0,0.3+0.6)--(0.1,0.2+0.6);
\draw[thick] (0.3,0.3+0.3)--(0,0+0.3);
\draw[thick] (0.2,0.1+0.3)--(0.3,0+0.3);
\draw[thick] (0,0.3+0.3)--(0.1,0.2+0.3);
\draw[->] (0.3,0.3)--(0,0);
\draw[thick] (0.3,0.3)--(0,0);
\draw[thick] (0.2,0.1)--(0.3,0);
\draw[->] (0.2,0.1)--(0.3,0);
\draw[thick] (0,0.3)--(0.1,0.2);
\end{tikzpicture}.
Since $\D$ has no lone crossings, there are at least two such ladders.

Let $\D^\prime$ be the diagram obtained from $\D$ by smoothing all the ladders joining $F_0$ and $F_1$. By the above observations, the diagram $\D$ can be obtained from $\D^\prime$ by shackle moves and crossing doublings.

Note that $\D^\prime$ is a positive diagram and has no nested Seifert circles or lone crossings. By the induction hypothesis, $\D^\prime$ can be obtained from a zero-crossing diagram by shackle moves and crossing doublings. Therefore, the diagram $\D$ can be obtained in the required way as well. The result follows.

\section{Elaboration on skein relation}

By definition, the HOMFLY-PT polynomial is a unique function that maps each oriented link diagram~$\D$ to a two-variable Laurent polynomial~$\P(\D) \in \mathbb{Z}[a^{\pm 1},z^{\pm 1}]$ such that:
\begin{enumerate}
\item[(i)] if two diagrams $\D$ and $\D^\prime$ represent the same link, then~$\P(\D) = \P(\D^\prime)$;
\item[(ii)] for any {\it skein triple} $(\D_+, \D_-, \D_0)$, i.e.~a triple of diagrams that coincide except within a small region where they differ as
(\begin{tikzpicture}[rotate=90, scale=0.9, baseline=1]
\draw[->] (0.3,0.3)--(0,0);
\draw[thick] (0.3,0.3)--(0,0);
\draw[thick] (0.2,0.1)--(0.3,0);
\draw[->] (0.2,0.1)--(0.3,0);
\draw[thick] (0,0.3)--(0.1,0.2);
\end{tikzpicture},
\begin{tikzpicture}[rotate=90, scale=0.9, baseline=1]
\draw[->] (0,0.3)--(0.3,0);
\draw[thick] (0,0.3)--(0.3,0);
\draw[thick] (0.1,0.1)--(0,0);
\draw[->] (0.1,0.1)--(0,0);
\draw[thick] (0.3,0.3)--(0.2,0.2);
\end{tikzpicture},
\begin{tikzpicture}[rotate=90, scale=0.9, baseline=1]
\draw[->] (0,0.3)--(0,0);
\draw[thick] (0,0.3)--(0,0);
\draw[->] (0.3,0.3)--(0.3,0);
\draw[thick] (0.3,0.3)--(0.3,0);
\end{tikzpicture}\hspace{-0.1cm}
),
the {\it skein relation} holds:
\begin{align*}
a^{-1} \P(\D_+) - a \P(\D_-) = z\P(\D_0);
\end{align*}
\item[(iii)] if $\D$ is a knot diagram 
with precisely zero crossings,
then~$\P(\D) = 1$.
\end{enumerate}

It is easy to show that if $\D$ has precisely zero crossings and $n$ link components, then
\begin{align}\label{HOMFLYPTProperties}
\P(\D) = (a^{-1}-a)^{n-1} z^{1 - n}.
\end{align}

\subsection{A formula for the HOMFLY-PT polynomial}

In practice, the following versions of the skein relation are more convenient:
\begin{align*}
\P(\D_+) &= a^{2} \P(\D_-) + a z\P(\D_0), \\
\P(\D_-) &= a^{-2} \P(\D_+) - a^{-1} z\P(\D_0).
\end{align*}

The process of calculating the HOMFLY-PT polynomial by repeated application of these relations, referred to as {\it resolution}, can be recorded schematically as shown in Fig. \ref{RTExamaple}.

\begin{figure}[h]
\centering
\includegraphics[width=10cm]{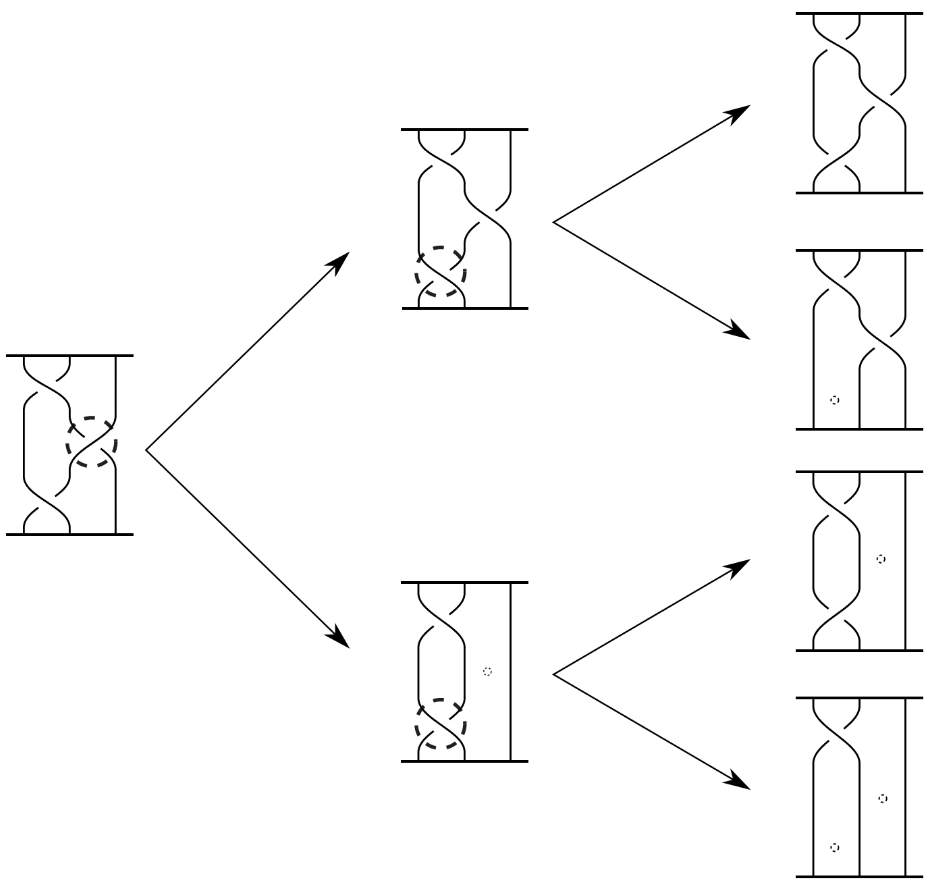}
\captionsource{A coherent resolution tree.}{Image credit: \cite{LDH19}.}
\label{RTExamaple}
\end{figure}

To be more precise, a {\it resolution tree} for $\D$ is a weighted rooted unitrivalent tree~$\T$ with a diagram associated with each node such that:
\begin{enumerate}
\item[(1)] the root node of $\T$ is $\D$;
\item[(2)] each triple (parent, left child, right child) is of the form $(\D_+, \D_0,\D_-)$ or $(\D_-,\D_0,\D_+)$;
\item[(3)] each branch of $\T$ is labelled with a monomial in $a$ and $z$ as shown in Fig. \ref{Resolving}.
\end{enumerate}

In other words, a resolution tree is a graph of the resolution branching process: at each internal node, one takes a crossing of the current diagram and branches by smoothing and flipping the crossing. For further purposes, the following additional condition is required:
\begin{enumerate}
\item[(4)] each crossing of $\D$ is resolved at most once.
\end{enumerate}

Naturally, any resolution tree $\T$ of $\D$ induces a decomposition of~$\P(\D)$ into a sum, indexed by leaf nodes of $\T$. Let us calculate the corresponding terms.

We refer to a one-to-one correspondence between leaf nodes of $\T$ and paths in $\T$ from the root to a leaf. Given a leaf node $\U$ of $\T$, the contribution of $\U$ to~$\P(\D)$ is $\P(\U)$ multiplied by the weights of the corresponding path.

As~Fig. \ref{Resolving} shows, the degree of $a$ in the weight of a branch equals the change of the writhe from the child to the parent. Besides, a $z$ term indicates that the child is obtained from the parent by a crossing smoothing, and a negative sign indicates that the smoothed crossing is negative. It follows that 
\begin{align}\label{HOMFLY-PT-Expansion}
\P(\D) = \sum\limits_{\U \in \T_\circ} (-1)^{t^\prime(\U)} z^{t(\U)} a^{\omega(\D) - \omega(\U)} \P(\U),
\end{align}
where $t(\U)$ is the number of crossings of $\D$ that one smoothed in obtaining~$\U$, $t^\prime(\U)$ is the number of negative crossings among the smoothed ones, and $\mathcal{T}_\circ$ is the set of leaf nodes of $\T$.

In particular, if the leaf nodes of $\T$ are unlinks, \eqref{HOMFLYPTProperties} allows calculating $\P(\D)$ explicitly.

\begin{figure}[h]
\centering
\includegraphics[width = 12.5cm]{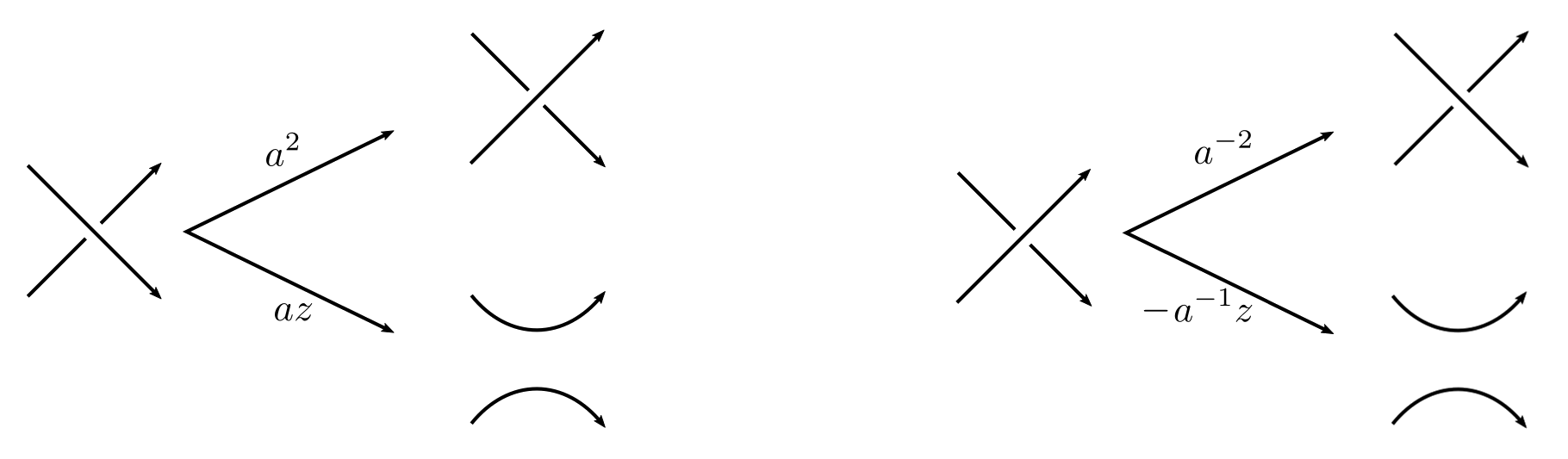}
\captionsource{The weight assignment for a resolution tree.}{Image credit: \cite{DHL19}.}
\label{Resolving}
\end{figure}

\subsection{Coherent resolution trees}

In this section, we introduce a class of resolution trees by adapting the construction of Y. Diao, G. Hetyei, and P. Liu \cite{DHL19}. We~call these trees {\it coherent} ones. At first, we need some definitions.

Let $p$ be an innermost point on $\D$. By the {\it maximal coherent path} starting at $p$ we mean the maximal descending path if the Seifert circle containing $p$ is loose on the left. Otherwise, if the Seifert circle is loose on the right, we mean the maximal ascending path. 

In other words, the precise meaning of such a path depends on the starting point. If the Seifert circle containing $p$ is loose on the left, we refer to the {\it descending rule}. Otherwise, we refer to the {\it ascending rule}.

\subsubsection{The construction}

Given a link diagram $\D$, let us describe all coherent resolution trees for $\D$ at once. The~construction consists of several phases. The resulting tree $\T$ depends on the choice of sequences of base points.

We begin at the first phase with the one node tree~$\T_0$. At~the end of phase $k$, we obtain a rooted subtree~$\T_k$ of~$\T$. The resulting subtrees satisfy $$\{\D\} = \T_0 \subseteq \T_1 \subseteq \ldots \subseteq \T_m = \T,$$ thus, each of the trees is obtained from the previous one by extensions shown in Fig.~\ref{Resolving}. 

The construction is inductive. Let $k \in \{1,2,\ldots,m\}$ be the index of another phase.

Let $\U_{k-1}$ be an arbitrary leaf node of $\T_{k-1}$. We suppose the diagram $\U_{k-1}$ contains $k-1$ base points~$x_1, x_2, \ldots, x_{k-1}$ lying on pairwise distinct link components of $\U_{k-1}$. We remove these components from consideration and choose an arbitrary appropriate\footnote{See Definition \ref{AppropriatePoints}.} point $x_k$ on the resulting diagram. We emphasize that this point depends on $\U_{k-1}$.

After the choice of the points $x_k$ is made for each leaf node $\U_{k-1}$, we do the following procedure. For each leaf node $\U_{k-1}$ of the current tree, we extend the tree at the crossing of~$\U_{k-1}$ which is the end of the maximal coherent path starting at $x_k$.

This procedure continues until, for each leaf node of the current tree (including the new ones), the maximal coherent path starting at $x_k$ is closed\footnote{The extensions terminate because the length of another maximal coherent path increases.}. After that, the $k$th phase ends.

The construction of the coherent tree $\T$ continues until, for each leaf node $\U$ of the current tree, one visits all link components of $\U$.

\begin{observation}\label{LeafsAreUnlinks}
For any leaf node $\U$ of $\T$, each link component of $\U$ is either descending or ascending. Besides, the link components of $\U$ are stacked over each other. Thus, $\U$ represents an unlink.
\end{observation}

By varying the appropriate base points on the leaf nodes of the trees $\T_0, \T_1, \ldots, \T_{m-1}$, we may obtain quite a few coherent resolution trees for $\D$.

\subsubsection{Properties of leafs}

Let $\T$ be an arbitrary coherent resolution tree for~$\D$. By Observation \ref{LeafsAreUnlinks}, each leaf $\U$ of $\T$ represents an unlink. Therefore, by \eqref{HOMFLY-PT-Expansion}, the term of $\P(\D)$ corresponding to $\U$ is
\begin{align}\label{HOMFLY-PT-Expansion1}
(-1)^{t^\prime(\U)} z^{t(\U)} a^{\omega(\D) - \omega(\U)} (a-a^{-1})^{\#\U - 1} z^{1 - \#\U}.
\end{align}

\begin{definition}
We say that a leaf $\U$ of $\T$ {\it contributes to the highest~$a$-degree term} if $$\omega(\D) - \omega(\U) + \#\U - 1 = \omega(\D) + s(\D)-1.$$
\end{definition}

The difference between our version of coherent resolution trees and the one of \cite{DHL19} is that our definitions of appropriate points differ. However, this is not essential for our purposes. In particular, the following result, which is Corollary 5.3 in \cite{DHL19}, is still applicable.

\begin{proposition}\label{a-HighestDegreeCriterion}
Let $\U$ be a leaf of $\T$. The following conditions are equivalent:
\begin{enumerate}
\item[(i)] $\U$ contributes to the highest~$a$-degree term;
\item[(ii)] the projection onto $\S$ of any link component of $\U$ is a simple closed curve.
\end{enumerate}
\end{proposition}

\section{Proof of Theorem \ref{MainTheorem}}

In this section, we address the bound \eqref{R} for positive diagrams. To begin with, we describe the behaviour of the sharpness under specific crossing smoothing.

\begin{lemma}\label{SkeinSharpnessLemma}
Let $(\U_+,\V,\U_0)$ be a triple of diagrams that coincide except within a small region where they differ as
(\begin{tikzpicture}[rotate=90, scale=0.9, baseline=1]
\draw[thick] (0.3,0.3+0.3)--(0,0+0.3);
\draw[thick] (0.2,0.1+0.3)--(0.3,0+0.3);
\draw[thick] (0,0.3+0.3)--(0.1,0.2+0.3);
\draw[->] (0.3,0.3)--(0,0);
\draw[thick] (0.3,0.3)--(0,0);
\draw[thick] (0.2,0.1)--(0.3,0);
\draw[->] (0.2,0.1)--(0.3,0);
\draw[thick] (0,0.3)--(0.1,0.2);
\end{tikzpicture},
\begin{tikzpicture}[rotate=90, scale=0.9, baseline=1]
\draw[->] (0,0.3)--(0,0);
\draw[thick] (0,0.3)--(0,0);
\draw[->] (0.3,0.3)--(0.3,0);
\draw[thick] (0.3,0.3)--(0.3,0);
\end{tikzpicture}, \hspace{-0.18cm}
\begin{tikzpicture}[rotate=90, scale=0.9, baseline=1]
\draw[->] (0.3,0.3)--(0,0);
\draw[thick] (0.3,0.3)--(0,0);
\draw[thick] (0.2,0.1)--(0.3,0);
\draw[->] (0.2,0.1)--(0.3,0);
\draw[thick] (0,0.3)--(0.1,0.2);
\end{tikzpicture}).
\begin{enumerate}
\item[(i)] If the bound \eqref{R} is sharp for $\U_+$, then it is sharp for at least one of $\V$ and $\U_0$.
\item[(ii)] If $\U_+$ is positive and the bound \eqref{R} is sharp for at least one of $\V$ and $\U_0$, then it is sharp for $\U_+$ as well.
\end{enumerate}
\end{lemma}
\begin{proof}
Denote by $\U_-$ the diagram obtained from $\U_+$ by replacing the region 
\begin{tikzpicture}[rotate=90, scale=0.9, baseline=1]
\draw[thick] (0.3,0.3+0.3)--(0,0+0.3);
\draw[thick] (0.2,0.1+0.3)--(0.3,0+0.3);
\draw[thick] (0,0.3+0.3)--(0.1,0.2+0.3);
\draw[->] (0.3,0.3)--(0,0);
\draw[thick] (0.3,0.3)--(0,0);
\draw[thick] (0.2,0.1)--(0.3,0);
\draw[->] (0.2,0.1)--(0.3,0);
\draw[thick] (0,0.3)--(0.1,0.2);
\end{tikzpicture}
with
\begin{tikzpicture}[rotate=90, scale=0.9, baseline=1]
\draw[->] (0.1,0.1)--(0,0);
\draw[thick] (0.3,0.3+0.3)--(0,0+0.3);
\draw[thick] (0.2,0.1+0.3)--(0.3,0+0.3);
\draw[->] (0.2,0.1)--(0.3,0);
\draw[thick] (0,0.3+0.3)--(0.1,0.2+0.3);
\draw[thick] (0,0.3)--(0.3,0);
\draw[thick] (0.1,0.1)--(0,0);
\draw[thick] (0.3,0.3)--(0.2,0.2);
\end{tikzpicture}. 
Since $\U_-$ and $\V$ represent the same link, one has $\P(\U_-) = \P(\V)$. 

By the skein relation,
\begin{align*}
\P(\U_+) = a^{2}\P(\U_-) + a z\P(\U_0) = a^{2}\P(\V) + a z\P(\U_0).
\end{align*}
Note that $\omega(\U_+) = \omega(\U_0) + 1 = \omega(\U_-) + 2$ and $s(\U_+) = s(\U_-) = s(\U_0)$. In particular, $$\omega(\U_+) + s(\U_+) - 1 = (\omega(\V) + s(\V) - 1) + 2 = (\omega(\U_0) + s(\U_0) - 1) + 1.$$ Therefore, the coefficient of the monomial $a^{\omega(\U_+) + s(\U_+) - 1}$ in $\P(\U_+)$ is $\Q_1(z) + \Q_2(z)$, where $\Q_1(z)$ and~$\Q_2(z)$ are, respectively, the coefficients of the monomials~$a^{\omega(\V) + s(\V) - 1}$ and~$a^{\omega(\U_0) + s(\U_0) - 1}$ in $\P(\V)$ and $\P(\U_0)$.

Note that the bound \eqref{R} is sharp for $\U_+$ if and only if $\Q_1(z) + \Q_2(z)$ is non-vanishing. Thus, if the bound is sharp, at least one of $\Q_1(z)$ and $\Q_2(z)$ is non-vanishing. Hence assertion~{\rm(i)} follows.

Next, suppose that $\U_+$ is positive. Then $\V$ and $\U_0$ are positive as well. We claim that the polynomials $\Q_1(z)$ and~$\Q_2(z)$ have non-negative coefficients. To prove this, let~$\T_1$ and~$\T_2$ be coherent resolution trees for, respectively, $\V$ and $\U_0$. Recall that by \eqref{HOMFLY-PT-Expansion1}, the term of~$\P$ corresponding to a leaf node $\W$ of a coherent resolution tree for a diagram $\D$ is
\begin{align*}
(-1)^{t^\prime(\W)} z^{t(\W)} a^{\omega(\D) - \omega(\W)} (a-a^{-1})^{\#\W - 1} z^{1 - \#\W},
\end{align*}
where $t(\W)$ is the number of crossings of $\D$ that one smoothed in obtaining~$\W$ and $t^\prime(\W)$ is the number of negative crossings among the smoothed ones. Note that if $\D$ is positive, then $t^\prime(\W) = 0$ for any leaf node $\W$. In particular, each monomial contributing to the highest~$a$-degree term has the same sign. Therefore, the claim follows.

Once again, the bound \eqref{R} is sharp for at least one of $\V$ and $\U_0$ if and only if at least one of the polynomials $\Q_1(z)$ and~$\Q_2(z)$ is non-vanishing. Thus, 
since these polynomials have non-negative coefficients, if at least one of $\V$ and $\U_0$ realize the bound, the sum $\Q_1(z) + \Q_2(z)$ is non-vanishing. Hence assertion {\rm(ii)} follows.
\end{proof}

Now, assume that the positive diagram $\D$ can be obtained from a zero-crossing diagram by shackle moves, crossing doublings, and Artin moves. To prove that the bound \eqref{R} is sharp for $\D$, we proceed by induction on its crossing number. If $\D$ has precisely zero crossings, the result follows from \eqref{HOMFLYPTProperties}. The inductive step follows from assertion {\rm (ii)} of Lemma \ref{SkeinSharpnessLemma}. Therefore, the result follows.

Finally, we deduce the opposite part of Theorem \ref{MainTheorem}. Assume that the bound \eqref{R} is sharp for $\D$. We aim to show that the diagram $\D$ can be obtained from a zero-crossing diagram by the required transformations. Again, we proceed by induction on the number of crossings of~$\D$. If it has precisely zero crossings, there is nothing to prove. Next, suppose that ${\rm Cr}(\D) > 0$.

First, let us state the following result.

\begin{lemma}\label{SharpNegativeHaveDoubleCrossing}
If the bound \eqref{R} is sharp for a positive diagram $\U$, then either~${\rm Cr}(\U) = 0$ or by using Artin moves, one can transform $\U$ into a diagram having a region of the form 
\begin{tikzpicture}[rotate=90, scale=0.9, baseline=1]
\draw[thick] (0.3,0.3+0.3)--(0,0+0.3);
\draw[thick] (0.2,0.1+0.3)--(0.3,0+0.3);
\draw[thick] (0,0.3+0.3)--(0.1,0.2+0.3);
\draw[->] (0.3,0.3)--(0,0);
\draw[thick] (0.3,0.3)--(0,0);
\draw[thick] (0.2,0.1)--(0.3,0);
\draw[->] (0.2,0.1)--(0.3,0);
\draw[thick] (0,0.3)--(0.1,0.2);
\end{tikzpicture}.
\end{lemma}

We prove this result in a separate section. At this moment, let us complete the proof of Theorem~\ref{MainTheorem}.

The above result implies that by using Artin moves, one can transform $\D$ into a diagram having a region of the form 
\begin{tikzpicture}[rotate=90, scale=0.9, baseline=1]
\draw[thick] (0.3,0.3+0.3)--(0,0+0.3);
\draw[thick] (0.2,0.1+0.3)--(0.3,0+0.3);
\draw[thick] (0,0.3+0.3)--(0.1,0.2+0.3);
\draw[->] (0.3,0.3)--(0,0);
\draw[thick] (0.3,0.3)--(0,0);
\draw[thick] (0.2,0.1)--(0.3,0);
\draw[->] (0.2,0.1)--(0.3,0);
\draw[thick] (0,0.3)--(0.1,0.2);
\end{tikzpicture}. 
Let $\D_0$ and $\V$ be the diagrams obtained from $\D$ by replacing this region with, respectively, 
\begin{tikzpicture}[rotate=90, scale=0.9, baseline=1]
\draw[->] (0.3,0.3)--(0,0);
\draw[thick] (0.3,0.3)--(0,0);
\draw[thick] (0.2,0.1)--(0.3,0);
\draw[->] (0.2,0.1)--(0.3,0);
\draw[thick] (0,0.3)--(0.1,0.2);
\end{tikzpicture}
and
\begin{tikzpicture}[rotate=90, scale=0.9, baseline=1]
\draw[->] (0,0.3)--(0,0);
\draw[thick] (0,0.3)--(0,0);
\draw[->] (0.3,0.3)--(0.3,0);
\draw[thick] (0.3,0.3)--(0.3,0);
\end{tikzpicture}. 
By assertion {\rm(i)} of Lemma \ref{SkeinSharpnessLemma}, the bound \eqref{R} is sharp for $\D_0$ or $\V$. Therefore, we can apply the induction hypothesis and conclude that $\D_0$ or $\V$ can be obtained from a zero-crossing diagram by shackle moves, crossing doublings, and Artin moves. Note that $\D$ is obtained 
from~$\D_0$ via a shackle move and 
from~$\V$ via a crossing doubling. 
Therefore, in any case, the diagram $\D$ can be obtained in the required way.

Therefore, Theorem \ref{MainTheorem} is proved.

\subsection{Proof of Lemma \ref{SharpNegativeHaveDoubleCrossing}}

Let $\D$ be a positive diagram. If ${\rm Cr}(\D) = 0$, there is nothing to prove. Suppose ${\rm Cr}(\D)>0$.

\begin{spacing}{1.5}
\end{spacing}

{\bf Step 1}. We remove all trivial components of $\D$, i.e.~those components of $\D$ that are Seifert circles. It is easy to check that such removal preserves the sharpness of the bound~\eqref{R}. Therefore, we can assume that $\D$ has no trivial components.

\begin{spacing}{1.5}
\end{spacing}

{\bf Step 2}. At this point, we are going to apply several Artin moves to $\D$. Let $x$ be an arbitrary appropriate\footnote{See Definition \ref{AppropriatePoints}.} point on $\D$, and denote by $R$ the Seifert circle containing $x$.

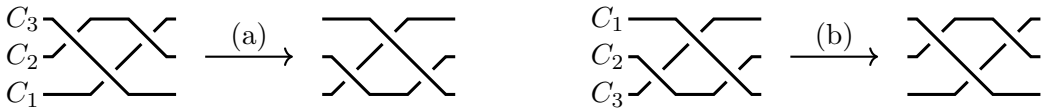
\begin{figure}[H]
\centering
\begin{tikzpicture}[rotate=90,scale=0.385, every node/.style={scale=0.5}]
\node[scale=2] at (0,0.7) {$C_1$};
\node[scale=2] at (1.3,0.7) {$C_2$};
\node[scale=2] at (2.6,0.7) {$C_3$};
\pic[
  rotate=90,
  line width=1.15pt,
  braid/control factor=0,
  braid/nudge factor=0,
  braid/gap=0.11,
  braid/number of strands = 3,
  name prefix=braid,
] at (0,0) {braid={
s_2^{-1}s_1^{-1}s_2^{-1}
}};
\node[scale=2] at (2,-0.64-3.9-1-1.5) {(a)};
\draw[thick, -To] (1.3,-0.64-3.9-1-0) -- (1.3,-0.64-3.9-2-2);
\pic[
  rotate=90,
  line width=1.15pt,
  braid/control factor=0,
  braid/nudge factor=0,
  braid/gap=0.11,
  braid/number of strands = 3,
  name prefix=braid,
] at (0,-0.64-3.9-3-2) {braid={
s_1^{-1}s_2^{-1}s_1^{-1}
}};
\node[scale=2] at (0,0.7-20) {$C_3$};
\node[scale=2] at (1.3,0.7-20) {$C_2$};
\node[scale=2] at (2.6,0.7-20) {$C_1$};
\pic[
  rotate=90,
  line width=1.15pt,
  braid/control factor=0,
  braid/nudge factor=0,
  braid/gap=0.11,
  braid/number of strands = 3,
  name prefix=braid,
] at (0,-20) {braid={
s_1^{-1}s_2^{-1}s_1^{-1}
}};
\node[scale=2] at (2,-0.64-3.9-1-1.5-20) {(b)};
\draw[thick, -To] (1.3,-0.64-3.9-1-0-20) -- (1.3,-0.64-3.9-2-2-20);
\pic[
  rotate=90,
  line width=1.15pt,
  braid/control factor=0,
  braid/nudge factor=0,
  braid/gap=0.11,
  braid/number of strands = 3,
  name prefix=braid,
] at (0,-0.64-3.9-3-2-20) {braid={
s_2^{-1}s_1^{-1}s_2^{-1}
}};
\end{tikzpicture}
\caption{Two directions of Artin moves.}
\label{NegativeSlide}
\end{figure}

Let $\M$ be the set of coherent\footnote{See Definition \ref{CoherentSequences}.} sequences of Seifert circles starting at $R$ and denote by~$[\M]$ the set of Seifert circles covered by such sequences. We assign a number~$\alpha(C)$ to each element~$C$ of~$[\M]$, namely, the distance to $R$ in the Seifert graph $\Gamma(\D)$. It is easy to check that for any triple $(C_1,C_2,C_3)$ of consecutive Seifert circles of a sequence from $\M$, the numbers assigned to them are consecutive as well.

At first, assume that the Seifert circle $R$ is loose on the right. In this case, for each triple~$(C_1,C_2,C_3)$ of consecutive Seifert circles from $[\M]$, we apply all possible Artin moves of type~(a) as shown in Fig.~\ref{NegativeSlide}.

\begin{claim}
The above process terminates.
\end{claim}
\begin{proof}
We assign the number $$\min(\alpha(C_1), \alpha(C_2)) \cdot {\rm Cr}(C_1,C_2)$$ to each pair $\{C_1,C_2\}$ of consecutive Seifert circles from $[\M]$, where ${\rm Cr}(C_1,C_2)$ denotes the number of crossings that $C_1$ and $C_2$ share. We aim to prove that each application of Artin move decreases by one the sum (over all the pairs) of assigned numbers.

Suppose a diagram $\D_2$ is obtained from $\D_1$ by such a transformation. Let $(C_1,C_2,C_3)$ be the corresponding triple of consecutive Seifert circles from $[\M]$. Since the Artin move is local, i.e.~it affects only a small region of the diagram, it is enough to compare the terms of the sum corresponding to those three Seifert circles. 

Note that the number of crossings that $C_1$ and $C_2$ share on $\D_2$ is one more than the one on~$\D_1$. Similarly, the number of crossings that $C_2$ and $C_3$ share on $\D_2$ is one less than the one on~$\D_1$. Besides, since $R$ is loose on the right and is coherent with each of $C_i$, we have $$\alpha(C_3) = \alpha(C_2) + 1 = \alpha(C_1) + 2.$$ Therefore, the contribution of the pairs of elements from $\{C_1,C_2,C_3\}$ to the sum for $\D_2$ is
\begin{align*}
&\alpha(C_1) \cdot ({\rm Cr}(C_1,C_2)+1) + \alpha(C_2) \cdot ({\rm Cr}(C_2,C_3)-1) = \\
&\alpha(C_1) \cdot {\rm Cr}(C_1,C_2) + \alpha(C_2) \cdot {\rm Cr}(C_2,C_3) + \alpha(C_1) - \alpha(C_2) = \\ 
&\alpha(C_1) \cdot {\rm Cr}(C_1,C_2) + \alpha(C_2) \cdot {\rm Cr}(C_2,C_3) - 1.
\end{align*}
The claim follows. 
\end{proof}

If the Seifert circle $R$ is loose on the left, we proceed similarly. Namely, we replace the Artin moves of direction (a) with those of direction (b). The key elements of the termination argument do not change, because in this case, the orientation of $\D$ forces the triple~$(C_1,C_2,C_3)$ to appear reversed\footnote{Also, see Fig.~\ref{LooseBrace1} below.} on the right of Fig \ref{NegativeSlide}. 

\begin{spacing}{1.5}
\end{spacing}

{\bf Step 3}. Let $\D^{\prime}$ be the diagram obtained from $\D$ by the above Artin moves. We aim to prove that~$\D^\prime$ contains a region of the form 
\begin{tikzpicture}[rotate=90, scale=0.9, baseline=1]
\draw[thick] (0.3,0.3+0.3)--(0,0+0.3);
\draw[thick] (0.2,0.1+0.3)--(0.3,0+0.3);
\draw[thick] (0,0.3+0.3)--(0.1,0.2+0.3);
\draw[->] (0.3,0.3)--(0,0);
\draw[thick] (0.3,0.3)--(0,0);
\draw[thick] (0.2,0.1)--(0.3,0);
\draw[->] (0.2,0.1)--(0.3,0);
\draw[thick] (0,0.3)--(0.1,0.2);
\end{tikzpicture}.
For this, we find a specific brace\footnote{See Definition \ref{Braces}.} of the castle~${\rm Cas}(x, \D^{\prime})$.

Note that the above-mentioned Artin moves may change the castle ${\rm Cas}(x, \D)$, i.e.~combinatorics of the ladders, the number of floors, et cetera. However, since they are local, they do not create traps. In particular, ${\rm Cas}(x, \D^{\prime})$ has no traps. It follows that one can choose $x$ to be the base point in the first phase of the construction of a coherent resolution tree for~$\D^\prime$. Let $\T$ be such a tree.

It is easy to see that Artin moves preserve the sharpness of the bound \eqref{R}. Since this bound is sharp for $\D$, it is sharp for $\D^{\prime}$ as well. Let $\U$ be an arbitrary leaf of $\T$ that contributes to the highest~$a$-degree term.

Denote by $\U_x$ the link component of $\U$ containing $x$. Let $F$ be the floor of the maximal level~$m$ that~$\U_x$ intersects.

By Step 1, the castle ${\rm Cas}(x;\D)$ has no less than one ladder. Since the diagram $\D^\prime$ is positive, $\U_x$ climbs at least one of them. Hence $m \geq 1$ and we can choose a ladder $s_1$ incident to $F$ that $\U_x$ climbs. By Proposition \ref{a-HighestDegreeCriterion}, the projection onto the two-sphere~$\S$ of~$\U_x$ is a simple closed curve. Hence the link component $\U_x$ is bounded within the castle and the Seifert circle $R$. Therefore, there exists another ladder incident to $F$ that $\U_x$ descends. Let~$s_2$ be the first ladder after $s_1$ incident to $F$ and $G$, where~$G$ is the unique floor of level~$m-1$ sharing ladders with $F$.

Denote by $F^\prime$ and $G^\prime$ the segments of $F$ and $G$ joining, respectively,~$s_1$ and~$s_2$. By the construction, the quadruple $(F^\prime, G^\prime, s_1, s_2)$ is a brace.

\begin{definition}
The {\it level} of a brace is the minimum of levels of its two segments. A brace is {\it loose from above} if there are no ladders on top of its segment of higher level.
\end{definition}

By the definition of $F$ and $s_1$, there are no ladders on top of $F^\prime$. Therefore, the brace~$(F^\prime, G^\prime, s_1, s_2)$ is loose from above. To complete the proof of Lemma \ref{SharpNegativeHaveDoubleCrossing}, it is enough to verify the following result.

\begin{claim}
Let $(F^\prime,G^\prime,s_1,s_2)$ be a\footnote{To make Fig. \ref{LooseBrace} and \ref{LooseBrace1} fit the marking, we abuse the notation and use the same symbols for the elements of the quadruple.} brace of level $k \geqslant 0$ that is loose from above. Either it has the form~\begin{tikzpicture}[rotate=90, scale=0.9, baseline=1]
\draw[thick] (0.3,0.3+0.3)--(0,0+0.3);
\draw[thick] (0.2,0.1+0.3)--(0.3,0+0.3);
\draw[thick] (0,0.3+0.3)--(0.1,0.2+0.3);
\draw[->] (0.3,0.3)--(0,0);
\draw[thick] (0.3,0.3)--(0,0);
\draw[thick] (0.2,0.1)--(0.3,0);
\draw[->] (0.2,0.1)--(0.3,0);
\draw[thick] (0,0.3)--(0.1,0.2);
\end{tikzpicture}
or $k\geqslant 1$ and the castle ${\rm Cas}(x, \D^{\prime})$ has a brace $(F^{\prime\prime},G^{\prime\prime},s_1^{\prime},s_2^{\prime})$ of level $k-1$ that is loose from above as well.
\end{claim}

\begin{figure}[h]
\centering
\begin{tikzpicture}[rotate=90,scale=0.6181175, every node/.style={scale=0.80275}]
\draw[ultra thick] (0,0+2.6) -- (0,-19.5);
\draw[ultra thick] (-1.3,0+2.6) -- (-1.3,-19.5);
\draw[ultra thick] (-2.6,0+2.6) -- (-2.6,-19.5);
\draw[very thick] (1.3-1.3,-1.3) -- (0-1.3,-1.3-1.3); 
\draw[ultra thick, draw=blue] (0-1.3,-1.3) -- (0.44-1.3,-1.3-0.44); 
\draw[ultra thick, draw=blue] (2.16-1.3-1.3,-1.3-1.3+0.44) -- (2.6-1.3+0.01-1.3,-1.3-1.3-0.01);
\draw[very thick] (1.3-2.6,-3.9) -- (0-2.6,-1.3-3.9); 
\draw[very thick] (0-2.6,-3.9) -- (0.44-2.6,-3.9-0.44); 
\draw[very thick] (2.16-1.3-2.6,-3.9-1.3+0.44) -- (2.6-1.3+0.01-2.6,-3.9-1.3-0.01);
\draw[very thick] (1.3-2.6,-6.5) -- (0-2.6,-1.3-6.5); 
\draw[very thick] (0-2.6,-6.5) -- (0.44-2.6,-6.5-0.44); 
\draw[very thick] (2.16-1.3-2.6,-6.5-1.3+0.44) -- (2.6-1.3+0.01-2.6,-6.5-1.3-0.01);
\draw[very thick] (1.3-2.6,-9.1) -- (0-2.6,-1.3-9.1); 
\draw[very thick] (0-2.6,-9.1) -- (0.44-2.6,-9.1-0.44); 
\draw[very thick] (2.16-1.3-2.6,-9.1-1.3+0.44) -- (2.6-1.3+0.01-2.6,-9.1-1.3-0.01);
\draw[very thick] (1.3-1.3,-11.7) -- (0-1.3,-11.7-1.3); 
\draw[very thick] (0-1.3,-11.7) -- (0.44-1.3,-11.7-0.44); 
\draw[very thick] (2.16-1.3-1.3,-11.7-1.3+0.44) -- (2.6-1.3+0.01-1.3,-11.7-1.3-0.01);
\draw[very thick] (1.3-1.3,-14.3) -- (0-1.3,-14.3-1.3); 
\draw[very thick] (0-1.3,-14.3) -- (0.44-1.3,-14.3-0.44); 
\draw[very thick] (2.16-1.3-1.3,-14.3-1.3+0.44) -- (2.6-1.3+0.01-1.3,-14.3-1.3-0.01);
\draw[ultra thick, draw=blue] (1.3-1.3,-16.9) -- (0-1.3,-16.9-1.3); 
\draw[very thick] (0-1.3,-16.9) -- (0.44-1.3,-16.9-0.44); 
\draw[very thick] (2.16-1.3-1.3,-16.9-1.3+0.44) -- (2.6-1.3+0.01-1.3,-16.9-1.3-0.01);
\draw[ultra thick, draw=blue] (0,-2.6) -- (0,-2.6-14.3);
\draw[ultra thick, draw=blue] (-1.3,0+2.6) -- (-1.3,-1.3);
\draw[ultra thick, draw=blue] (-1.3,-19.5+1.3) -- (-1.3,-19.5);
\node[scale=1] at (-0.65,-1.3) {$s_1$};
\node[scale=1] at (-0.65,-1.3-10.4) {$s_2$};
\node[scale=1] at (0,0.5+2.6) {$F$};
\node[scale=1] at (-1.3,0.5+2.6) {$G$};
\node[scale=1] at (-2.6,0.5+2.6) {$H$};
\node[scale=1] at (-0.65-1.3,-1.3-2.6) {$s_1^\prime$};
\node[scale=1] at (-0.65-1.3,-1.3-2.6-2.6) {$s_2^\prime$};
\node[scale=1] at (0.3,-7.15) {$F^\prime$};
\node[scale=1] at (0.3-1.3,-7.15) {$G^\prime$};
\node[scale=1] at (0,-19.5-1-0.39) {$k+1$};
\node[scale=1] at (-1.3,-19.5-1) {$k$};
\node[scale=1] at (-2.6,-19.5-1-0.39) {$k-1$};
\node[scale=1, text=blue] at (0.5,-7.15-9.1) {$\U_x$};
\draw [blue,-{Straight Barb}, thick] (0,-7.8) -- (0,-7.8-0.65); 
\draw [-{Straight Barb}, thick] (-1.3,-7.8) -- (-1.3,-7.8-0.65); 
\draw [-{Straight Barb}, thick] (-2.6,-7.8) -- (-2.6,-7.8-0.65); 
\end{tikzpicture}
\caption{Finding a brace if the Seifert circle $R$ is loose on the right.}
\label{LooseBrace}
\end{figure}
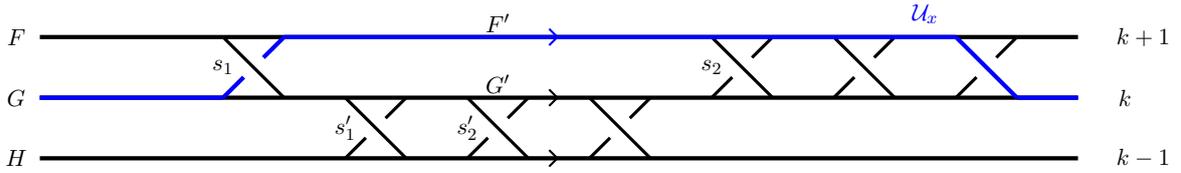

\begin{proof}
Since the castle has no traps, there are no ladders underneath $F^\prime$ or on top of $G^\prime$. If the segment $G^\prime$ is incident to no ladders, the brace has the form 
\begin{tikzpicture}[rotate=90, scale=0.9, baseline=1]
\draw[thick] (0.3,0.3+0.3)--(0,0+0.3);
\draw[thick] (0.2,0.1+0.3)--(0.3,0+0.3);
\draw[thick] (0,0.3+0.3)--(0.1,0.2+0.3);
\draw[->] (0.3,0.3)--(0,0);
\draw[thick] (0.3,0.3)--(0,0);
\draw[thick] (0.2,0.1)--(0.3,0);
\draw[->] (0.2,0.1)--(0.3,0);
\draw[thick] (0,0.3)--(0.1,0.2);
\end{tikzpicture}, and there is nothing to prove. Suppose that $G^\prime$ is incident to at least one ladder. 

Since the brace $(F^\prime,G^\prime,s_1,s_2)$ is not a trap, the ladders incident to $G^\prime$ lie on the opposite side of~$G^\prime$ from~$s_1$ and $s_2$. In particular, the level $k$ of $G^\prime$ is not zero.

Recall that $G$ denotes the floor containing the segment $G^\prime$. Denote by $H$ the unique floor of level~$k-1$ sharing ladders with $G$. First, we prove that any ladder incident to $G^\prime$ is incident to $H$. 

Let $t$ be any such ladder. Following the orientation of~$G$, let us order all the ladders incident to~$G$. Let $t_1$ be the last ladder that $G$ and $H$ share located on the left of $t$, and let~$t_2$ be the first ladder that $G$ and $H$ share located on the right of $t$ (by the definition of a floor, the sets of such ladders are non-empty). The ladders $t_1$ and $t_2$ are adjacent (among all the ladders that $G$ and $H$ share), thus, determine a brace. Since this brace is not a trap,~$t = t_1$ or~$t = t_2$. Therefore, $t$ is incident to $H$.

Second, we prove that $G^\prime$ and $H$ share at least two ladders. By Observation \ref{NoTrapsImpliesCoherence}, the triple of consecutive Seifert circles corresponding to $H$, $G^\prime$, and $F^\prime$ is a part of a sequence from $\M$. In particular, by Step 2, the ladders of ${\rm Cas}(x, \D^{\prime})$ corresponding to that triple do not admit Artin moves. Therefore, the result follows.

Third, let $s_1^\prime$ and $s_2^\prime$ be the first two ladders that $G^\prime$ and $H$ share. Fig. \ref{LooseBrace} and \ref{LooseBrace1} show that they determine a brace of level $k-1$. Since $(F^\prime,G^\prime,s_1,s_2)$ is not a trap, that brace is loose from above. Therefore, the claim follows.
\end{proof}

\begin{figure}[h]
\centering
\begin{tikzpicture}[rotate=90,scale=0.6181175, every node/.style={scale=0.80275}]
\draw[ultra thick] (0,0+2.6) -- (0,-19.5);
\draw[ultra thick] (-1.3,0+2.6) -- (-1.3,-19.5);
\draw[ultra thick] (-2.6,0+2.6) -- (-2.6,-19.5);
\draw[ultra thick,draw=blue] (1.3-1.3-1.3,-1.3) -- (0-1.3-1.3,-1.3-1.3); 
\draw[very thick] (0-1.3-1.3,-1.3) -- (0.44-1.3-1.3,-1.3-0.44); 
\draw[very thick] (2.16-1.3-1.3-1.3,-1.3-1.3+0.44) -- (2.6-1.3+0.01-1.3-1.3,-1.3-1.3-0.01);
\draw[very thick] (1.3-2.6+1.3,-3.9) -- (0-2.6+1.3,-1.3-3.9); 
\draw[very thick] (0-2.6+1.3,-3.9) -- (0.44-2.6+1.3,-3.9-0.44); 
\draw[very thick] (2.16-1.3-2.6+1.3,-3.9-1.3+0.44) -- (2.6-1.3+0.01-2.6+1.3,-3.9-1.3-0.01);
\draw[very thick] (1.3-2.6+1.3,-6.5) -- (0-2.6+1.3,-1.3-6.5); 
\draw[very thick] (0-2.6+1.3,-6.5) -- (0.44-2.6+1.3,-6.5-0.44); 
\draw[very thick] (2.16-1.3-2.6+1.3,-6.5-1.3+0.44) -- (2.6-1.3+0.01-2.6+1.3,-6.5-1.3-0.01);
\draw[very thick] (1.3-2.6+1.3,-9.1) -- (0-2.6+1.3,-1.3-9.1); 
\draw[very thick] (0-2.6+1.3,-9.1) -- (0.44-2.6+1.3,-9.1-0.44); 
\draw[very thick] (2.16-1.3-2.6+1.3,-9.1-1.3+0.44) -- (2.6-1.3+0.01-2.6+1.3,-9.1-1.3-0.01);
\draw[very thick] (1.3-1.3-1.3,-11.7) -- (0-1.3-1.3,-11.7-1.3); 
\draw[very thick] (0-1.3-1.3,-11.7) -- (0.44-1.3-1.3,-11.7-0.44); 
\draw[very thick] (2.16-1.3-1.3-1.3,-11.7-1.3+0.44) -- (2.6-1.3+0.01-1.3-1.3,-11.7-1.3-0.01);
\draw[very thick] (1.3-1.3-1.3,-14.3) -- (0-1.3-1.3,-14.3-1.3); 
\draw[very thick] (0-1.3-1.3,-14.3) -- (0.44-1.3-1.3,-14.3-0.44); 
\draw[very thick] (2.16-1.3-1.3-1.3,-14.3-1.3+0.44) -- (2.6-1.3+0.01-1.3-1.3,-14.3-1.3-0.01);
\draw[very thick] (1.3-1.3-1.3,-16.9) -- (0-1.3-1.3,-16.9-1.3); 
\draw[ultra thick,draw=blue] (0-1.3-1.3,-16.9) -- (0.44-1.3-1.3,-16.9-0.44); 
\draw[ultra thick,draw=blue] (2.16-1.3-1.3-1.3,-16.9-1.3+0.44) -- (2.6-1.3+0.01-1.3-1.3,-16.9-1.3-0.01);
\draw[ultra thick, draw=blue] (-2.6,-2.6) -- (-2.6,-2.6-14.3);
\draw[ultra thick, draw=blue] (-1.3,0+2.6) -- (-1.3,-1.3);
\draw[ultra thick, draw=blue] (-1.3,-19.5+1.3) -- (-1.3,-19.5);
\node[scale=1] at (-0.65-1.3,-1.3) {$s_1$};
\node[scale=1] at (-0.65-1.3,-1.3-10.4) {$s_2$};
\node[scale=1] at (0,0.5+2.6) {$H$};
\node[scale=1] at (-1.3,0.5+2.6) {$G$};
\node[scale=1] at (-2.6,0.5+2.6) {$F$};
\node[scale=1] at (-0.65,-1.3-2.6) {$s_1^\prime$};
\node[scale=1] at (-0.65,-1.3-2.6-2.6) {$s_2^\prime$};
\node[scale=1] at (-2.6-0.3,-7.15) {$F^\prime$};
\node[scale=1] at (-0.3-1.3,-7.15) {$G^\prime$};
\node[scale=1] at (0,-19.5-1-0.39) {$k-1$};
\node[scale=1] at (-1.3,-19.5-1) {$k$};
\node[scale=1] at (-2.6,-19.5-1-0.39) {$k+1$};
\node[scale=1, text=blue] at (-2.6-0.5,-7.15-9.1) {$\U_x$};
\node[scale=1, text=white] at (0.5,-7.15-9.1) {The author is supported by Pavel Solikov};
\draw [-{Straight Barb}, thick] (0,-7.8) -- (0,-7.8-0.65); 
\draw [-{Straight Barb}, thick] (-1.3,-7.8) -- (-1.3,-7.8-0.65); 
\draw [blue,-{Straight Barb}, thick] (-2.6,-7.8) -- (-2.6,-7.8-0.65); 
\end{tikzpicture}
\caption{Finding a brace if the Seifert circle $R$ is loose on the left.}
\label{LooseBrace1}
\end{figure}
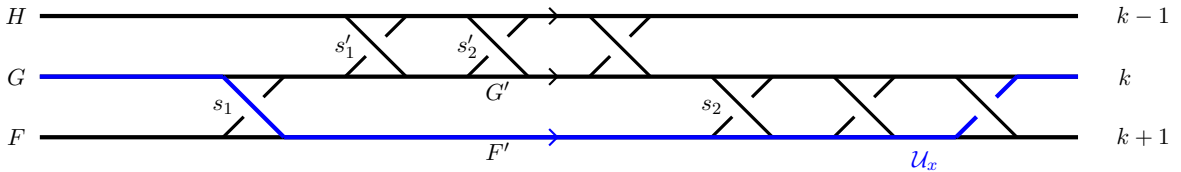


\begin{thebibliography}{GMM14}
\bibitem[Ban22]{Ban22} I. Banfield, {\it Almost all strongly quasipositive braid closures are fibered}, Journal of Knot Theory and its Ramifications 31(11) (2022).
\bibitem[BZ03]{BZ03} G. Burde, H. Zieschang, {\it Knots}, 2nd ed., de Gruyter Studies in Mathematics, vol. 5, Walter de Gruyter \& Co., Berlin, 2003.
\bibitem[Cro89]{Cro89} P. R. Cromwell, {\it Homogeneous links}, Journal of the London Mathematical Society, 2(3):535–552, 1989.
\bibitem[Di04]{Di04} Y. Diao, {\it The additivity of the crossing number}, J Knot Theory Ramif 13, 7 (2004), 857–866.
\bibitem[DHL19]{DHL19} Y. Diao, G. Hetyei, P. Liu, {\it The braid index of reduced alternating links}, Math. Proc. Cambridge Philos. Soc. 2019.
\bibitem[Etn05]{Etn05} J. Etnyre, {\it Legendrian and transversal knots}, in: Handbook of Knot Theory, Elsevier, Amsterdam, 2005, pp. 105–185.
\bibitem[FW87]{FW87} J. Franks, R. F. Williams, {\it Braids and the Jones polynomial}, Trans. Amer. Math. Soc., 303(1):97–108, 1987.
\bibitem[GMM14]{GMM14} J. Gonz\`{a}lez-Meneses, P. M. G. Manch\`{o}n, {\it Closures of positive braids and the Morton-Franks-Williams inequality}, Topology and its Applications, Volume 174, 2014, Pages 14-24, ISSN 0166-8641.
\bibitem[Gru03]{Gru03} H. Gruber, {\it Estimates for the minimal crossing number}, \url{arxiv.org/abs/math/0303273v3}.
\bibitem[HS15]{HS15} K. Hayden, J. M. Sabloff, {\it Positive knots and Lagrangian fillability}, Proc. Amer. Math. Soc., 143(4):1813–1821, 2015.
\bibitem[Hos05]{Hos05} J. Hoste, {\it The enumeration and classification of knots and links}, in: Handbook of Knot Theory, Elsevier, Amsterdam, 2005, pp. 209–232.
\bibitem[Kau87]{Kau87} L. Kauffman, {\it State models and the Jones polynomial}, Topology 26(3) (1987), 395–407.
\bibitem[LM87]{LM87} W. B. R. Lickorish, K. C. Millett, {\it A polynomial invariant of oriented links}, Topology, Volume 26, Issue 1, 1987, Pages 107-141, ISSN 0040-9383.
\bibitem[LDH19]{LDH19} P. Liu, Y. Diao, G. Hetyei, {\it The Homfly polynomial of links in closed braid form}, Discrete Mathematics, Volume 342, Issue 1, 2019, Pages 190-200, ISSN 0012-365X.
\bibitem[Mor86]{Mor86} H. R. Morton, {\it Seifert circles and knot polynomials}, Math. Proc. Cambridge Philos. Soc., 99(1):107–109, 1986.
\bibitem[Mur87]{Mur87} K. Murasugi, {\it The Jones Polynomial and Classical Conjectures in Knot Theory}, Topology 26 (1987): 187-194.
\bibitem[Mur91]{Mur91} K. Murasugi, {\it On the braid index of alternating links}, Trans. Amer. Math. Soc. 326 (1) (1991), 237–260.
\bibitem[Nak21]{Nak21} K. Nakagane, {\it A full-twisting formula for the HOMFLY polynomial}, Pacific J. Math., 313(1):185–193, 2021.
\bibitem[Rud05]{Rud05} L. Rudolph, {\it Knot theory of complex plane curves}, in: Handbook of Knot Theory, Elsevier, Amsterdam, 2005, pp. 349-427.
\bibitem[St05]{St05} A. Stoimenow, {\it On polynomials and surfaces of variously positive links}, J. Eur. Math. Soc. 7(4), 477–509 (2005).
\bibitem[Thi87]{Thi87} M. Thistlethwaite, {\it A Spanning Tree Expansion of the Jones Polynomial}, Topology 26(3) (1987), 297–309.
\bibitem[Thi88]{Thi88} M. B. Thistlethwaite, {\it On the Kauffman polynomial of an adequate link}, Invent. Math. 93, no. 2 (1988): 285–96.
\bibitem[Tur87]{Tur87} V. Turaev, {\it A simple proof of the Murasugi and Kauffman theorems on alternating links}, Enseign. Math. (2) 33 (1987), 203–225.
\bibitem[Yam87]{Yam87} S. Yamada, {\it The minimal number of Seifert circles equals the braid index of link}, Invent. Math. 891 (1987), 347–356.
\bibitem[Yok92]{Yok92} Y. Yokota, {\it Polynomial invariants of positive links}, Topology 31(4) (1992), 805–811.
\end{thebibliography}
\end{document}